\documentclass[12pt]{amsart}
\usepackage{geometry}                
\usepackage{graphicx,bbm}
\usepackage{amssymb}
\usepackage{epstopdf}
\usepackage{setspace}
\usepackage{comment,subcaption,tikz-cd}
\usepackage{tikz}

\usepackage{todonotes}

\DeclareMathOperator{\kernel}{ker}

\DeclareMathOperator{\im}{Im}

\DeclareMathOperator{\vol}{Vol}

\DeclareMathOperator{\PH}{PH}

\def\R{\mathbb{R}}
\def\F{\mathbb{F}}
\def\Z{\mathbb{Z}}

\def\birth{\mathrm{birth}}
\def\death{\mathrm{death}}
\def\cC{\mathcal{C}}

\def\cF{\mathcal{F}}

\def\cM{\mathcal{M}}

\def\cP{\mathcal{P}}

\def\cR{\mathcal{R}}
\def\cU{\mathcal{U}}
\def\cS{\mathcal{S}}

\def\cY{\mathcal{Y}}

\newcommand{\E}{\mathbb{E}} 

\newcommand{\mean}[1] {\E\left\{{#1}\right\}}

\newcommand{\ind}{\boldsymbol{\mathbbm{1}}} 

\newcommand{\set}[1]{\left\{#1\right\}}

\newcommand{\norm}[1]{\left\|#1\right\|}
\newcommand{\param}[1]{\left(#1\right)}
\newcommand{\abs}[1] {\left| {#1}\right|}
\newcommand{\floor}[1] {\left\lfloor{#1}\right\rfloor}

\newcommand{\prob}[1]{\mathbb{P}\left(#1\right)}

\newcommand{\eps}{\epsilon}

\providecommand{\setthms}[1]{#1}
\setthms{
\newtheorem{lem}{Lemma}[section]
\newtheorem{thm}[lem]{Theorem}

\theoremstyle{definition}
\newtheorem{defn}[lem]{Definition}

}
\newcommand{\cech}{\v{C}ech }

\newcommand{\iid}{\mathrm{i.i.d.}}

\newcommand{\ninf}{n\to\infty}

\newcommand{\limninf}{\lim_{\ninf}}

\numberwithin{equation}{section}

\def\ve{\vec{e}}

\def\ve{\varepsilon}

\newcommand{\nodraw}[1]{}

\title{Maximally persistent cycles in random geometric complexes}

\author{Omer Bobrowski}
\address{Omer Bobrowski - Duke University, Department of Mathematics}
\email{omer@math.duke.edu}
\thanks{O.B.\ gratefully acknowledges the support NSF DMS-1418261 and NSF DMS-1209155.}

\author{Matthew Kahle}
\address{Matthew Kahle - The Ohio State University, Department of Mathematics}
\email{mkahle@math.osu.edu}
\thanks{M.K.\ gratefully acknowledges support from DARPA \#N66001-12-1-4226, NSF \#CCF-1017182, NSF \#DMS-1352386, the Institute for Mathematics and its Applications, and the Alfred P.\ Sloan Foundation}

\author{Primoz Skraba}
\address{Primoz Skraba - Jozef Stefan Institute, A.I. Laboratory\\Ljubljana, Slovenia}
\email{primoz.skraba@ijs.si}
\thanks{P.S.\ was funded  by the EU project TOPOSYS (FP7-ICT-318493-STREP)}

\date{\today}                                   

\usepackage{thmtools,thm-restate}

\begin{document}


\begin{abstract}

We initiate the study of persistent homology of random geometric simplicial complexes. Our main interest is in maximally persistent cycles of degree-$k$ in  persistent homology, for a either the \cech or the Vietoris--Rips filtration built on a  uniform Poisson process of intensity $n$ in the unit cube $[0,1]^d$. This is a natural way of measuring the largest ``$k$-dimensional hole"  in a random point set.
This problem is in the intersection of geometric probability and algebraic topology, and is naturally motivated by a probabilistic view of topological inference.

We show that for all $d \ge 2$ and $1 \le k \le d-1$ the maximally persistent cycle has (multiplicative) persistence of order
$$ \Theta \left( \left( \frac{\log n}{\log \log n} \right)^{1/k} \right),$$
with high probability, characterizing its rate of growth as $n \to \infty$. The implied constants depend on $k$, $d$, and on whether we consider the Vietoris--Rips or \cech filtration.

\end{abstract}


\maketitle


\section{Introduction}


The study of topological properties of random graphs has a long history, dating back to classical results on the connectivity, cycles, and largest components in Erd\H{o}s--Renyi graphs \cite{ER59,ER60}. Generalizations have been developed in several directions. One direction is to consider different models of random graphs (see, e.g.~\cite{bollobas2003mathematical,Penrose}). Another direction is to consider higher-dimensional topological properties, resulting in the study of \emph{random simplicial complexes} rather than random graphs, where in addition to vertices and edges the structure consists also of triangles, tetrahedra and higher dimensional simplexes (see, e.g.~\cite{aronshtam2015does,clique,flag, linial2006homological}). The study of random simplicial complexes focuses mainly on their homology, which is a natural generalization of the notions of connected components and cycles in graphs. Homology is an algebraic topology framework that is used to study cycles in various dimensions, where (loosely speaking) a $k$-dimensional cycle can be thought of as the boundary of a $k+1$ dimensional solid (see Section \ref{sec:back} for more details).

In random \emph{geometric} simplicial complexes,  the vertices are generated by a random point process (e.g. Poisson) in a metric space, and then geometric conditions are applied to determine which of the simplexes should be included in the complex. The two most studied models are the random \cech and Vietoris-Rips complexes (see Section \ref{sec:back} for definitions).
Several recent papers have studied various aspects of the topology of these complexes (see \cite{bobrowski_distance_2011, BM14,bob_vanishing,geometric,Meckes, YA12, YSA14} and the survey \cite{BK14}). These papers contain theorems which characterize the phase transitions where homology appears and disappears, estimates for the Betti numbers (the number of $k$-dimensional cycles), limiting distributions, etc. While this line of research presents a  deep and interesting theory, it is also motivated by data analysis applications.

Topological data analysis (TDA) is a recently emerging field that focuses on extracting topological features from sampled data, and uses them as an input for various data analytic and statistical algorithms. The main idea behind it is that topological properties could help us understand the structure underlying the data, and provide us with a set of features that are robust to various types of deformations (cf. \cite{Carlsson09,carlsson2008local,Ghrist08}). Geometric complexes play a key role in computing topological features from a finite set of data points. The construction of these complexes usually depends on one or more parameters (e.g. radius of balls drawn around the sample points), and the ability to properly extract topological features depends on choosing this parameter correctly. One of the most powerful tools in TDA is a multi-scale version of homology, called \emph{persistent homology} (see Section \ref{sec:back}), which was developed mainly to solve this problem of sensitive parameter tuning. 
In persistent homology, instead of finding the best parameter values, one considers the entire range of  possible values. As the parameter values change, the observed topological features change (e.g.~cycles are created and filled in). Persistent homology tracks these changes and provides a way to measure the significance of the features that show up in this process. One way to represent the information provided by persistent homology is via \emph{barcodes}, see Figure \ref{fig:ph_example}. Here, every bar corresponds to a feature in the data and its endpoints correspond to the times (parameter value) where the feature was created and terminated.
The underlying philosophy in TDA is that topological features that survive (or persist) through a long range of parameter values are significant and related to real topological structures in the data (or the ``topological signal"), whereas ones with a shorter lifespan are artifacts of the finite sampling, and correspond to noise (see \cite{Wass14}). This approach motivates the following question: How long does a ``long range" of parameters (or a long bar in the barcode) have to be in order to be considered significant? 
Phrased differently - how long should we expect this range to be, if the sample points were entirely random, without any underlying structure or features? 
This is the main question we try to answer in this paper. 

To be more specific, in this paper we study the case where the data points are generated by a homogeneous Poisson process in the unit $d$-dimensional cube $[0,1]^d$ ($d>1$) with intensity $n$, denoted by $\cP_n$. We consider the persistent homology of both the \cech complex $\cC(\cP_n,r)$ and the Rips complex $\cR(\cP_n,r)$, where the scale parameter $r$ is the radius of the balls used to create these complexes (see Section \ref{sec:back}). We denote by $\Pi_k(n)$ the maximal persistence of a cycle in the degree $k$ persistent homology ($1\le k \le d-1$) of either the \cech or the Rips complex. Note that $\Pi_k(n)$ is a property of the persistent homology, where we consider all possible radii, and therefore it does not depend $r$.
Our main result shows that, with high probability,
\[
	\Pi_k(n) \sim \param{\frac{\log n}{\log\log n}}^{1/k},
\]
in the sense that $\Pi_k(n)$ can be bounded from above and below by a matching term up to a constant factor.
The precise definitions and statements are presented in Section \ref{sec:main}.
The proofs for the upper and lower bounds require very different techniques. To prove the upper bound we present a novel `isoperimetric-type' statement (Lemma \ref{lem:min_comp}) that links the persistence of a cycle to the number of vertices that are used to form it. The lower bound proof uses an exhaustive search for a specific construction that guarantees the creation of a persistent cycle.

In addition to proving the theoretical result,  in Section \ref{sec:experimental} we also present extensive numerical experiments confirming the computed bounds and empirically computing the implied constants. These results also suggest a conjectural law of large numbers. Finally, we note that while the results in this paper are presented for the homogeneous Poisson process on a $d$-dimensional cube, they should hold with minor adjustments also to non-homogenous processes as well as for shapes other than the cube. We also predict that our statements will hold for more generic point processes (e.g.~weakly sub-Poisson processes), using some of the statements made in \cite{YA12}. The detailed analysis of these more generic cases is left as future work. 

\vspace{10pt}

\textbf{Earlier work:}\ 
The study of the topology of random geometric complexes has been growing rapidly in the past decade. Most of the results so far are related to homology rather than persistent homology (i.e.~fixing the parameter value).
The study in \cite{bob_vanishing,geometric} focuses mainly on the phase transitions for appearance and vanishing of homology, which can be viewed as higher dimensional generalizations of the phase transition for connectivity in random  graphs. In \cite{bobrowski_distance_2011,BM14,Meckes,YSA14} more emphasis was given to the distribution of the Betti numbers, namely the number of cycles that appear. Similar questions for more general point processes have also been considered in \cite{YA12}. In \cite{adler2013crackle,owada2015limit} simplicial complexes generated by distributions with an unbounded support were studied from an extreme value theory perspective.   The recent survey \cite{BK14} overviews recent progress in this area.

The study of random persistent homology, on the other hand, is at its very initial stages.
Recall that the $0$-th homology represent the connected components in a space. Thus, the results in \cite{appel_connectivity_2002,penrose_longest_1997} about the connectivity threshold in random geometric graphs could be viewed as related to the $0$-th persistence homology of either the \cech or the Rips complex.
The first study of persistent homology in degree $k\ge 1$ for a random setting was for $n$ points chosen uniformly i.i.d.\ on a circle by Bubenik and Kim \cite{BK07}. In this setting, they used the theory of order statistics to describe the limiting distribution of the persistence diagram.
Another direction of study is the persistence diagrams of random functions. In \cite{omereuler}, the authors study the ``persistent Euler characteristic" of Gaussian random fields.

Another line of research (see e.g.~\cite{bobrowski_topological_2014,  chazal2013bootstrap,ChazalFLRW14,ChazalGLM14,chazal2011scalar,chazal2013persistence,Wass14}) focuses on statistical inference using  persistent homology, and include results about confidence intervals, consistency and robustness for topological estimation, subsampling and bootstrapping methods, and more.

Finally, we point out the earlier work in geometric probability \cite{BGAS13}, measuring the largest convex hole for a set of random points in a convex planar region $R$.  A convex hole is generated when there is a subset of points for which the convex hull is empty (i.e.~contains no other points from the set). The size of a convex hole is then measured combinatorially, as the number of vertices generating the hole.  In \cite{BGAS13} it is shown that the largest hole has $\Theta \left( \log n / \log \log n \right)$ vertices, regardless of the shape of the ambient convex region $R$. In this paper we are also measuring the size of the largest hole, but in a very different sense. We are using the algebraic-topological notion of holes (via persistent homology), rather than combinatorial notion of counting vertices, so as far as we can tell the fact that these two ways of measuring the size of the largest hole have the same right of growth (when $d=2$ and $k=1$) is something of a coincidence.


As far as we know, this article presents the first detailed probabilistic analysis for persistent $k$th homology of random geometric complexes, for $k\ge 1$.

\vspace{10pt}
The structure of the paper is as follows. In Section \ref{sec:back} we provide the topological and probabilistic building blocks we will use throughout the paper. In Section \ref{sec:main} we present the main result - the asymptotic behavior of maximally persistent cycles.
In Sections \ref{sec:upper} and \ref{sec:lower} we provide the main parts of the proof for the random \cech complex (upper and lower bounds, respectively). Some parts of the proofs require more knowledge in algebraic topology than the others, and we present those in Section \ref{sec:topo} (including the proof for the Rips complex). Finally, in Section \ref{sec:experimental} we present simulation results, complementing the main (asymptotic) result of the paper.


\section{Background}\label{sec:back}


In this section we provide a brief introduction to the topological and probabilistic notions used in this paper. 


\subsection{Homology}\label{sec:homology}


We wish to introduce the concept of homology here in an intuitive rather than in a rigorous way. For a comprehensive introduction to homology, see \cite{Hatcher} or \cite{munkres_elements_1984}.
Let $X$ be a topological space, and $\F$ a field. The \textit{homology of $X$ with coefficients in $\F$} is a set of vector spaces $\set{H_k(X)}_{k=0}^\infty$, which are topological invariants of $X$ (i.e.~they are invariant under homeomorphisms).
We note that the  standard notation is $H_k(X, \F)$ where $\F$ denotes the coefficient ring, but we  suppress the field and let $H_k(X)$ denote homology with $\F$ 
coefficients throughout this article.
 
The dimension of the zeroth homology $H_0(X)$ is equal to the number of connected components of $X$. For $k\ge 1$, the basis elements of the $k$-th homology $H_k(X)$ correspond to $k$-dimensional ``holes" or (nontrivial-) ``cycles" in $X$. An intuitive way to think about a $k$-dimensional cycle is as the result of taking the boundary of a $(k+1)$-dimensional body.
For example, if $X$ a circle then $H_0(X) \cong \F$, and $H_1(X) \cong \F$. If $X$ is a $2$-dimensional sphere then $H_0(X)\cong \F$ and $H_2(X) \cong \F$, while $H_1(X)\cong\set{0}$ (since every loop on the sphere can be shrunk to a point). In general if $X$ is a $n$-dimensional sphere, then
\[
H_k(X) \cong \begin{cases} \F & k=0,n \\
0 & \mbox{otherwise}.
\end{cases}
\]
%
%

We will use $H_*(X)$ when making a statement that applies to all the homology groups simultaneously. In addition to providing information about spaces, homology is also used to study mappings between spaces. If $f: X\to Y$ is a map between two topological spaces, then it induces a map in homology $f_*:H_*(X)\to H_*(Y)$. This map is a linear transformation between vector spaces which tells us how cycles in $X$ map to cycles in $Y$. These mappings are important when discussing persistent homology.

Finally, we say that two spaces $X,Y$ are \emph{homotopy equivalent}, denoted by $X\simeq Y$, if $X$ can be continuously deformed to $Y$ (loosely speaking). In particular, if $X\simeq Y$ then $H_*(X)\cong H_*(Y)$ (isomorphic). For example, a circle, an empty triangle and an annulus are all homotopy equivalent.


\subsection{The \cech and Vietoris-Rips complexes}


As mentioned earlier, the \cech and the Rips complexes are often used to extract topological information from data. These complexes are abstract simplicial complexes \cite{Hatcher} and in our case will be generated by a set of points in $\R^d$.
These complexes are tied together with the union of balls we define as
\begin{equation}\label{eq:union_balls}
	\cU(\cP, r) = \bigcup_{p\in \cP}B_r(p),
\end{equation}
where $\cP\subset \R^d$, and $B_r(p)$ is a $d$-dimensional ball of radius $r$ around $p$. Note that the set $\cP$ does not have to be discrete, in which case we can think of $\cU(\cP,r)$ as a ``tube" around $\cP$. The definitions of the complexes are as follows.


\begin{defn}[\cech complex]\label{def:cech_complex}
Let $\cP = \set{x_1,x_2,\ldots,x_n}$ be a collection of points in $\R^d$, and let $r>0$. The \cech complex $\cC(\cP, r)$ is constructed as follows:
\begin{enumerate}
\item The $0$-simplices (vertices) are the points in $\cP$.
\item A $k$-simplex $[x_{i_0},\ldots,x_{i_k}]$ is in $\cC(\cP,r)$ if $\bigcap_{j=0}^{k} {B_{r}(x_{i_j})} \ne \emptyset$.
\end{enumerate}
\end{defn}


\begin{defn}[Vietoris--Rips complex]\label{def:rips_complex}
Let $\cP = \set{x_1,x_2,\ldots,x_n}$ be a collection of points in $\R^d$, and let $r>0$. The Vietoris--Rips complex $\cR(\cP, r)$ is constructed as follows:
\begin{enumerate}
\item The $0$-simplices (vertices) are the points in $\cP$.
\item A $k$-simplex $[x_{i_0},\ldots,x_{i_k}]$ is in $\cR(\cP,r)$ if $B_r(x_{i_j})\cap B_r(x_{i_l}) \ne \emptyset$ for all $0\le j,l \le k$.
\end{enumerate}
\end{defn}


Note that the Rips complex $\cR(\cP,r)$ is the flag (or clique) complex built on top of the geometric graph $G(\cP,2r)$, where two vertices $x_i,x_j$ are connected if and only if $\norm{x_i-x_j} \le 2r$.
The difference between the \cech and the Rips complexes, is that for the \cech complex we require all $k+1$ balls to intersect in order to include a face, whereas for the Rips complex we only require pairwise intersections between the balls.
Figure \ref{fig:cech} shows an example for the \cech and Rips complexes constructed from the same set of points and the same radius $r$, and highlights this difference.

Part of the importance of the \cech complex stems from the following statement known as the ``Nerve Lemma" (see \cite{borsuk_imbedding_1948}). We note that the original lemma is more general then stated here, but we will only be using it in the following special case,
\begin{lem}\label{lem:nerve}
Let $\cP\subset \R^d$ be a finite set of points.
Then $\cC(\cP,r)$ is homotopy equivalent to $\cU(\cP,r),$
and in particular
\[
	H_*(\cC(\cP,r)) \cong H_*(\cU(\cP,r)).
\]
\end{lem}


The Rips complex is commonly used in applications, as in some practical cases it requires less computational resources. In an arbitrary metric space, using the triangle inequality we have the following inclusions of complexes,
\begin{equation}\label{eq:rips_cech}
	\cC(\cP,r) \subset \cR(\cP,r) \subset \cC(\cP,2r).
\end{equation}
For subsets of Euclidean space, the constant $2$ can be improved (see  \cite{dsG07}).


\begin{figure}[h!]
\centering
  \includegraphics[scale=0.35]{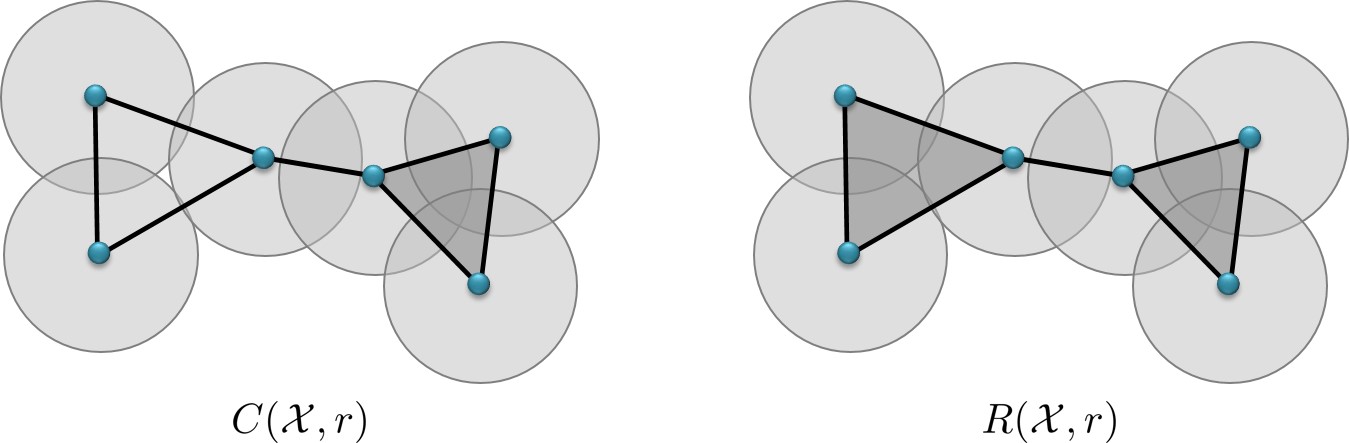}
\caption{On the left - the \cech complex $\cC(\cP,r)$, on the right - the Rips complex $\cR(\cP,r)$ with the same set of vertices and the same radius. We see that the three left-most balls do not have a common intersection and therefore do not generate a 2-dimensional face in the \cech complex. However, since all the pairwise intersections occur, the Rips complex does include the corresponding face.}
\label{fig:cech}
\end{figure}


\subsection{Persistent homology}


Let $\cP\subset \R^d$, and consider the following indexed sets -
\[
	\cU := \set{\cU(\cP,r)}_{r=0}^\infty,\quad \cC := \set{\cC(\cP,r)}_{r=0}^\infty,\quad\cR := \set{\cR(\cP,r)}_{r=0}^\infty.
\]
These three sets are examples of `filtrations' - nested sequences of sets, in the sense that $\cF_{r_1} \subset \cF_{r_2}$ if $r_1 < r_2$ (where $\cF$ is either $\cU,\cC,\textrm{or }\cR$).

As the parameter $r$ increases, the homology of the spaces $\cF_r$ may change. The \emph{persistent homology} of $\cF$, denoted by $\PH_*(\mathcal{F})$, keeps track of this process. Briefly, $\PH_k(\cF)$ contains  information about the $k$-th homology of the individual spaces $\cF_r$ as well as the mappings between the homology of $\cF_{r_{1}}$ and $\cF_{r_{2}}$ for every $r_{1} < r_{2}$ (induced by the inclusion map). The \emph{birth time}
of an element (a cycle) in $\PH_*(\cF)$ can be thought of as the value of $r$ where this element appears for the first time. 
The \emph{death time} is the value of $r$ where an element vanishes, or merges with another existing element. 

Formally, we consider a filtration with parameter values from $[0,\infty)$, the birth and death times can be defined as:
\begin{defn}\label{def:birth} The \emph{birth} of an element $\gamma\in \PH_k(\cF)$ is 
$$\gamma_{birth} := \min\set{r:  \gamma \in H_k(X_{r}) }$$ 
\end{defn}
\begin{defn}\label{def:death} The \emph{death time} of an element $\gamma\in \PH_k(\cF)$ is 
$$\gamma_{death}:=\min \set{r: \gamma \in \kernel(H_k(X_{\gamma_\birth}) \rightarrow H_k(X_{r}))}$$ 
\end{defn}

One useful way to describe persistent homology is via 
the notion of \textit{barcodes} \cite{Ghrist08}. 
A barcode for the persistent homology of a filtration $\cF$
is  a collection of graphs, one for each order of homology group. A
bar in the $k$-th graph, starting at $b$ and ending at $d$ ($b\le d$) indicates the existence of an element of $\PH_k(\cF)$ (or a $k$-cycle) whose birth and death times are $b$ and $d$ respectively. 
In Figure \ref{fig:ph_example} we present the barcode for the filtration $\cU$ where $\cP$ is a set of $50$ random points lying inside an annulus. The intuition is that the longest bars in the barcode represent  ``true" features in the data (e.g.~the connected component and the $1$-cycle in the annulus), whereas the short bars are regarded to as ``noise."
It can be shown that the pairing between birth and death times is sufficient to yield a unique barcode~\cite{zomorodian2005computing}. 

\begin{figure}[h!]
\centering
\begin{subfigure}[b]{0.75\textwidth}
{
     {\includegraphics[width=\textwidth]{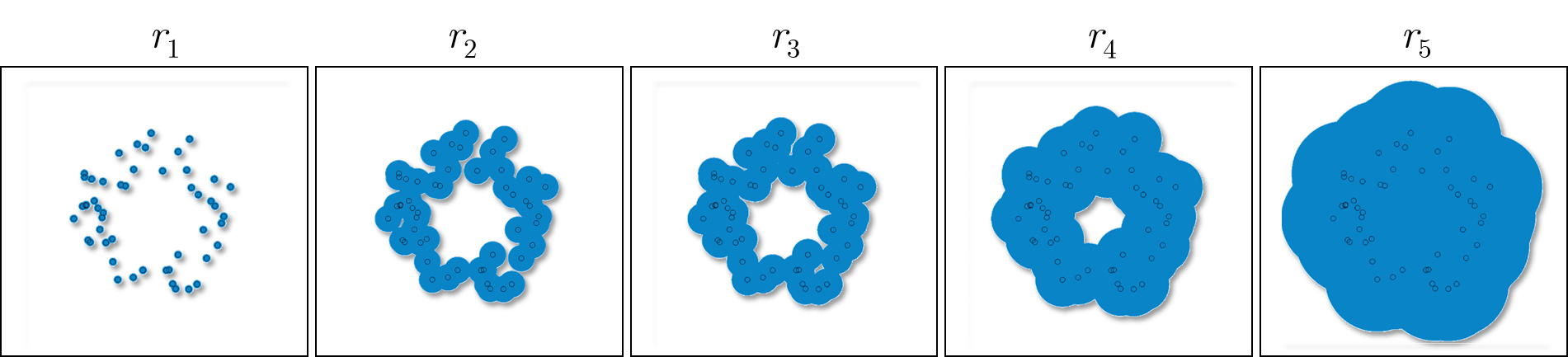}}
}
\end{subfigure}

\begin{subfigure}[b]{0.75\textwidth}

{
   {\includegraphics[width=\textwidth]{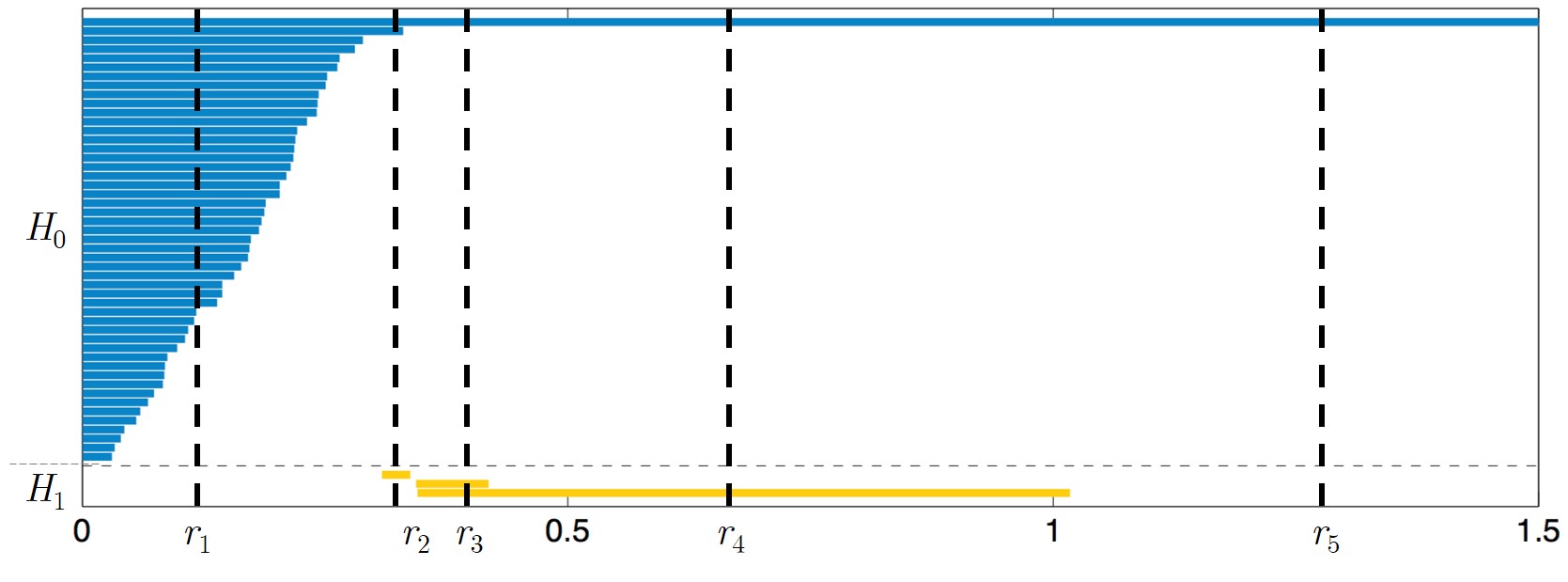}}
}
\end{subfigure}
\caption{(a) $\cF_r = \cU_r$ is a union of balls of radius $r$ around $\cP$ - a random set of $n=50$ points, uniformly distributed on an annulus in $\R^2$. We present five snapshots of this filtration.
(b) The persistent homology of the filtration $\cF$. The $x$-axis is the radius $r$, and the bars represent the cycles that born and die. For $H_0$ we observe that at radius zero the number of components is exactly $n$ and as the radius increases components merge (or die). The $1$-cycles show up later in this process. There are two bars that are significantly longer than the others (one in $H_0$ and one in $H_1$). These correspond to the true features of the annulus.  }
\label{fig:ph_example}
\end{figure}


\subsection{The Poisson process}


In this paper, the set of points we use to construct either a \cech or a Rips complex will be generated by a Poisson process $\cP_n$, which can be defined as follows.
Let $X_1,X_2,\ldots$ be an infinite sequence of $\iid$ (independent and identically distributed) random variables in $\R^d$. We will focus on the case where $X_i$ is uniformly distributed on the unit cube ${\mathcal Q}^d= [0,1]^d$. We note, however, that our results hold (with minor adjustments) for any distribution with a compact support and density bounded above and below.
Next, fix $n>0$, take $N \sim \mathrm{Poisson}({n})$, independent of the $X_i$'s, and define
\begin{equation}\label{eq:def_pp}
\cP_n = \set{X_1,X_2,\ldots, X_N}.
\end{equation}
Two properties characterizing the Poisson process $\cP_n$ are:
\begin{enumerate}
\item For every Borel-measurable set $A\subset \R^d$ we have that
$$
	\abs{\cP_n \cap A} \sim \mathrm{Poisson}({n \vol(A\cap {\mathcal Q}^d)}),
$$
where $\abs{\cdot}$ stands for the set cardinality, and $\vol(\cdot)$ is the Lebesgue measure.
\item If $A,B\subset\R^d$ are disjoint sets then $\abs{\cP_n\cap A}$ and $\abs{\cP_n \cap B}$ are independent random variables (this property is known as `spatial independence').
\end{enumerate}
The Poisson process $\cP_n$ is closely related to the fixed-size set $\set{X_1,\ldots,X_n}$. Note that the expected number of points in $\cP_n$ is $\mean{N} = n$. In fact, most results known for one of these processes apply to the other with very minor, or no, changes. This is true for the results presented in this paper as well. However we choose to focus only on $\cP_n$, mainly due to its spatial independence property. 

In the following we  study asymptotic phenomena, when $n\to\infty$. In this context, if $E_n$ is an event that depends on $n$, we say that $E_n$ occurs \emph{with high probability} (w.h.p.) if $\limninf\prob{E_n} = 1$.


\section{Maximally persistent cycles}\label{sec:main}

For the remainder of this paper assume that $d \ge 2$ and $1 \le k \le d-1$ are fixed.
Let $\cP_n$ be the Poisson process defined above, and define
\[
\cU(n,r) := \cU(\cP_n,r),\quad \cC(n,r) := \cC(\cP_n,r),\quad \cR(n,r) := \cR(\cP_n,r).
\] 
Let $\PH_k(n)$ be the $k$-th persistent homology of either of the filtrations for $\cU,\cC,\textrm{ or }\cR$ (it will be clear from the context which filtration we are looking at). Note that from the Nerve Lemma (\ref{lem:nerve}) we have that  $\cU(n,r)\simeq\cC(n,r)$, so we will state the results only for $\cC$ and $\cR$. However, some of the statements we make are easier to prove for the balls in $\cU$ rather than the simplexes in $\cC$, and we shall do so.

For every cycle $\gamma \in \PH_k(n)$ we denote by $\gamma_{birth},\gamma_{death}$ the birth and death times (radii) of $\gamma$, respectively. Commonly (see \cite{Carlsson09,Ghrist08}), the \emph{persistence} of a cycle is measured by 
the length of the corresponding bar in the barcode, namely by the difference $\delta(\gamma) := \gamma_{death}-\gamma_{birth}$. In this paper, however, we choose to define the persistence of $\gamma$ in a multiplicative way as
\begin{equation}\label{eq:pers_multi}
	\pi(\gamma) := \frac{\gamma_{death}}{\gamma_{birth}}.
\end{equation}

There are several reasons for defining the persistence of a cycle this way.
\begin{itemize}
\item This definition is equivalent to saying that we measure the difference in a logarithmic scale. Studying persistent homology in the logarithmic scale is common \cite{chazal2013persistence,buchet_efficient_2015,hudson_mesh_2009,phillips_geometric_2013,sheehy_persistent_2014}.

\item  This definition is scale invariant, which is desirable, since `topological significance' should focus on shape rather than size. 
For example, consider the cycles corresponding to $\gamma_1,\gamma_2$ in Figure \ref{fig:multi}. These two cycles are created by exactly the same configuration of points, just at a different scale. Therefore, we would like to say that these cycles are equally significant. Clearly, $\delta(\gamma_1) > \delta(\gamma_2)$, while $\pi(\gamma_1) = \pi(\gamma_2)$. Thus, our definition works better in this case.

In addition, this scale invariance guarantees that a linear change in the units used to measure the data (e.g.~from inches to cm, or from degrees Celsius to Fahrenheit) will not affect the persistence value.

\item One purpose of using a persistence measure is to differentiate between cycles that capture phenomena underlying the data, and those who are created merely due to chance. To this end, the `physical size' of the cycle is not necessarily the correct measure. Consider, for example, the cycles corresponding to $\gamma_2$ and $\gamma_3$ in Figure \ref{fig:multi}. 
Intuitively, we would like to claim that $\gamma_2$ is more significant than $\gamma_3$, as the former is created by a very `stable' configuration of points, while the latter is created by outliers that clearly tell us nothing about the underlying structure. In this example, taking the `additive' persistence we will have that $\delta(\gamma_2) < \delta(\gamma_3)$, simply because the overall size of the annulus is much smaller than that of the triangle. However, taking multiplicative persistence yields $\pi(\gamma_2) > \pi(\gamma_3)$, which is more consistent with our intuition.

\item Both the \cech and Vietoris--Rips complexes are important in TDA, and the natural relationship between these complexes is a multiplicative one (see \eqref{eq:rips_cech}). Because of this relationship, our results hold for both random \cech and Rips complexes, up to a constant factor (see Section \ref{sec:rips}). 
Indeed, the majority of approximation results for geometric complexes are multiplicative~\cite{sheehy2013linear,chazal2014persistence,dey2015graph}, making multiplicative persistence more relevant to existing stability guarantees. 

\item The argument from Section 5 of this paper suggests that there are many cycles $\gamma$ for which $\gamma_{birth} = o(\gamma_{death})$. In this case, it is hard to differentiate between cycles by looking at $\gamma_{death}-\gamma_{birth} \approx \gamma_{death}$.

\end{itemize}

\begin{figure}[h!]
\centering
  \includegraphics[scale=0.25]{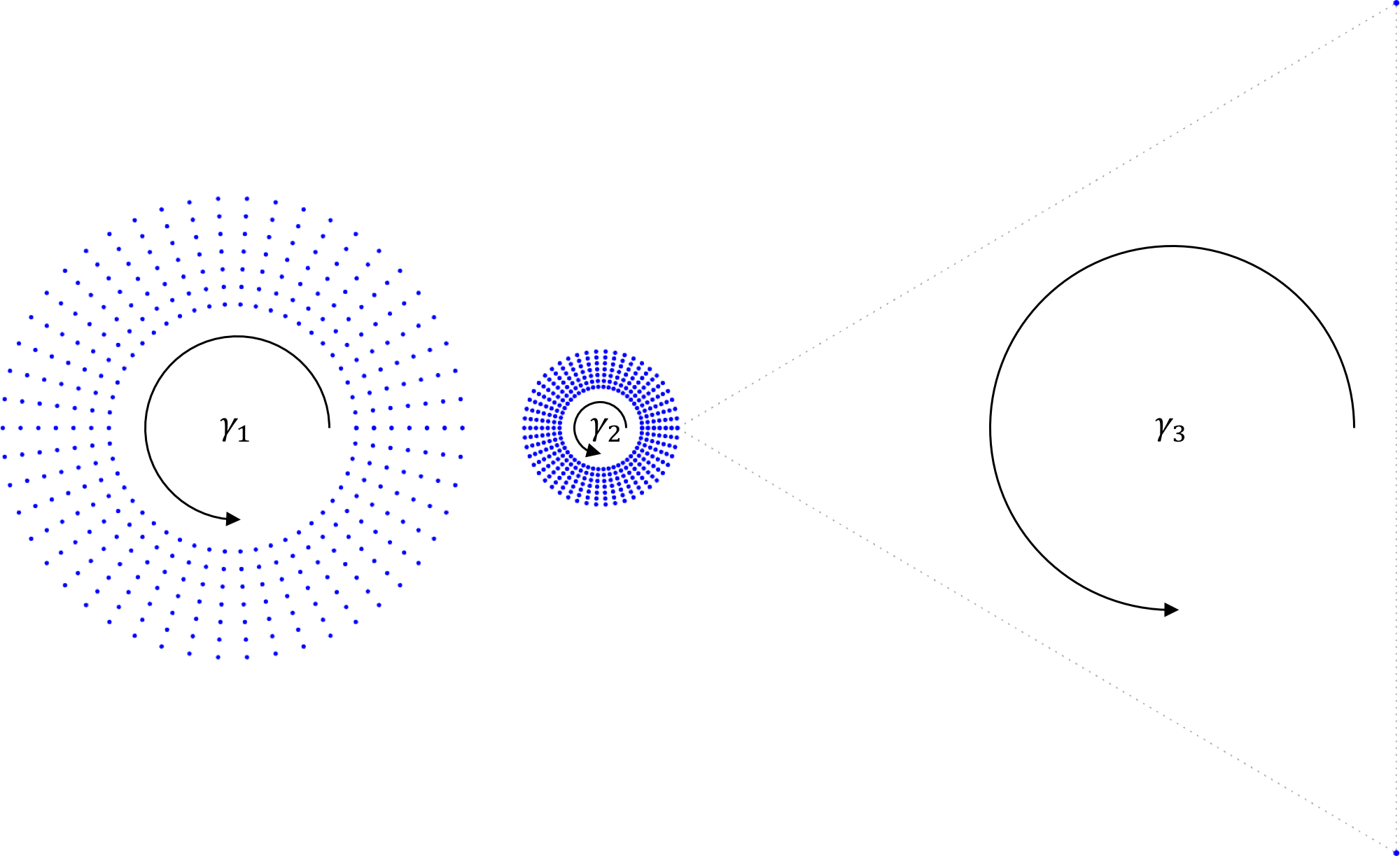}
\caption{Multiplicative persistence as a significance measure. The dataset in this example consists of a few hundred points sampled from two annuli, and two outliers (on the right). We are interested in the $1$-cycles that denoted by $\gamma_1,\gamma_2,\gamma_3$, that correspond to the two annuli and the triangle on the right.}
\label{fig:multi}
\end{figure}

Our main interest is in the maximal persistence over all $k$-cycles, defined as
\begin{equation}\label{eq:def_max_pers}
	\Pi_k(n) := \max_{\gamma \in \PH_k(n)}\pi(\gamma).
\end{equation}
More specifically, we are interested in the asymptotic behavior of $\Pi_k(n)$ as $n\to \infty$. 
The main result in this paper is that $\Pi_k(n)$ scales like the function $\Delta_k(n)$, defined by
\begin{equation}\label{eq:delta_n}
\Delta_k(n) :=  \param{\frac{\log n}{\log \log n}}^{1/k}.
\end{equation}
In particular we have the following theorem.


\begin{thm}\label{thm:main}
For fixed $d \ge 2$, and $1\le k \le d-1$, 
let $\cP_n$ be a Poisson process on the unit cube $[0,1]^d$ defined in \eqref{eq:def_pp}, and let $\PH_k(n)$ be the $k$-th persistent homology of either $\cC,\textrm{or } \cR$.
Then there exist positive constants $A_k, B_k$ such that
\[
	\limninf \prob{A_k \le \frac{\Pi_k(n)}{\Delta_k(n)} \le B_k} = 1.
\]
\end{thm}


{\bf Remarks:} 
\begin{enumerate}

\item The constants $A_k$ and $B_k$ depend on $k$ (the homology degree), $d$ (the ambient dimension), and on whether we consider the \cech or the Rips complex. 
We conjecture that a law of large numbers holds, namely that $\Pi_k(n)/\Delta_k(n) \to C_k$ for some $A_k \le C_k \le B_k$. For some evidence for this conjecture, see the experimental results in Section \ref{sec:experimental}.
In the following sections we will prove Theorem \ref{thm:main}.

\item The additive persistence $\delta(\gamma)$ can be bounded naively by the result on the contractibility of the \cech complex in \cite{geometric}. More concretely, Theorem 6.1 states that if $r \ge c \left(\frac{\log n}{n}\right)^{1/d} $ then the \cech complex is contractible (w.h.p.). This implies that for every cycle $\gamma$ we have $\delta(\gamma) \le \gamma_\death \le c \left(\frac{\log n}{n}\right)^{1/d}$. Similar statements can be made about $\PH_0$ using the connectivity radius in \cite{appel_connectivity_2002,penrose_longest_1997} (which is of the same $(\log n/n)^{1/d}$ scale). However, these are only crude upper bounds on the additive persistence, 
that do not differentiate between the different cycles in persistent homology, or even between different degrees of homology (note that these bounds do not depend on $k$).

\item The study in \cite{geometric} suggests the following upper bound for $\Pi_k(n)$. As mentioned before, we know that $\gamma_{\death} \le c \param{\frac{\log n}{n}}^{1/d}$ for all $\gamma$. In addition, the analysis in \cite{geometric} shows that if $n^{k+1}r^{dk} \to 0$ then $H_k(\cC(n,r)) = 0$, which implies that $\gamma_{\birth} \ge c' n^{-\frac{k+2}{d(k+1)}}$ for some $c'>0$.
Therefore, we have that $\pi(\gamma) = O\param{{(\log n)^{1/d}  n^{\frac{1}{d(k+1)}}}}$. However, as we shall see later, this is a very crude upper bound.

\end{enumerate}


\section{Proof - Upper Bound}\label{sec:upper}


For this section and the next one,  consider the \cech complex only. We want to prove the upper bound in Theorem \ref{thm:main}. That is, we need to show that there exists a constant $B_k>0$ depending only on $k$ and $d$,
so that with high probability 
$$ \Pi_k(n) \le B_k  \Delta_k(n) = B_k\left( \frac{\log n }{\log \log n} \right)^{1/k}.$$

The main idea in proving the upper bound in Theorem \ref{thm:main} is to show that large cycles require the formation of a large connected component in $\cC(n,r)$ at a very early stage (small radius $r$). To this end we will provide two bounds: (1) a lower bound for the size of the connected component supporting a large cycle (Lemma \ref{lem:min_comp}), and (2) an upper bound for the size of connected components in $\cC(n,r)$ for small values of $r$ (Lemma \ref{lem:max_comp}). 


\begin{restatable}{lem}{lemmincomp}\label{lem:min_comp}
Let $\gamma\in \PH_k(n)$, with $\gamma_{birth} = r$ and $\pi(\gamma) = p$. Then there exists a constant $C_1$ such that $\cC(n,r)$ contains a connected component with at least $m = C_1 p^k$ vertices. The constant $C_1$  depends on $k,d$ only.
\end{restatable}


The proof for this lemma requires more working knowledge in algebraic topology than the rest of this paper, and we defer it to Section \ref{sec:topo}. At this point, we would like to suggest an intuitive explanation. Suppose that $\cC(n,r)$ contains a $k$-cycle such that all the points generating it lie on a $k$-dimensional sphere of radius $R$, and such that there are no points of $\cP_n$ inside the sphere. In that case the death time of the cycle will be $R$ and then $\pi(\gamma) = p \ge R/r $. The minimum number of balls of radius $r$ required to cover a $k$-dimensional sphere of radius $R$ is known as the ``covering number" and is proportional to $(R/r)^k = p^k$. The cycle created is then a part of a connected component of $\cC(n,r)$ containing at least $C\times p^k$ vertices. Intuitively, creating a cycle with the same birth and death times in any other way (i.e.~not necessarily around a sphere) will require coverage of an area larger than the $k$-dimensional sphere, and therefore larger connected components. To make this statement  precise,  in Section \ref{sec:topo} we present an  isoperimetric-type inequality for $k$-cycles. Note that this statement is completely deterministic (i.e.~non-random).

The following lemma bounds the number of vertices in a connected component of the \cech complex $\cC(n,r)$, for small values of $r$.


\begin{lem} \label{lem:max_comp}
Let $\alpha > 0$ be fixed. There exists a constant $C_2 > 0$ depending only on $\alpha$ and $d$ such that if
$$nr^d \le \frac{C_2}{(\log n)^{\alpha}}$$
 and $$m \ge \alpha^{-1}\frac{\log n}{\log\log n},$$ then
with high probability  $\cC(n,r)$ has no connected components with more than $m$ vertices.
\end{lem}


\begin{proof}[Proof of Lemma \ref{lem:max_comp}]
Let $N_m(r)$ be the number of subsets of $\cP_n$ with  $m$ vertices, that are connected in $\cC(n,r)$.
We can write $N_m(r)$ as 
\[
	\sum_{\cY\subset \cP_n} \ind\set{\cC(\cY,r) \textrm{ is connected}},
\] 
where the sum is over all sets $\cY$ of $m$ vertices. We will show that choosing $r$ and $m$ as the lemma states, we have $\prob{N_m(r) > 0} \to 0$ which implies the statement of the lemma.

By Palm theory (see for example, Theorem 1.6 of \cite{Penrose}) we have that
\[
	\mean{N_m(r)}  = \frac{n^m}{m!} \prob{ \cC(\{X_1,\ldots, X_m\},r) \textrm{ is connected}},
	\]
where $X_i \sim U([0,1]^d)$ are $\iid$ variables. If $\cC(\{X_1,\ldots, X_m\},r)$  is connected, then the underlying graph must contain a subgraph isomorphic to a tree on $m$ vertices.
Suppose that $\Gamma$ is a labelled tree on the vertices $\set{1,\ldots,m}$. Assuming that  vertex $1$ is the root, for $2\le i\le m$ let $\mathrm{par}(i)$ be the parent of vertex $i$ in the tree. Suppose also that the vertices are ordered so that $\mathrm{par}(i) < i$. If $\cC(\{X_1,\ldots, X_m\},r)$ contains  $\Gamma$ then every $X_i$ must be connected to $X_{\mathrm{par}(i)}$ which implies that $X_i \in B_{2r}(X_{\mathrm{par}(i)})$. Therefore,
\[\begin{split}
	\prob{ \cC(\{X_1,\ldots, X_m\},r) \textrm{ contains $\Gamma$}} &\le \prob{X_i \in B_{2r}(X_{\mathrm{par}(i)}),\ \forall 2\le i \le m}\\
	&\le \int_{[0,1]^d}\int_{B_{2r}(x_{\mathrm{par}(2))}}\cdots \int_{B_{2r}(x_{\mathrm{par}(m)})} dx_m\cdots dx_1\\
	&= (\omega_d 2^dr^d)^{m-1}.
\end{split}
\]
The second inequality is due to the effect of the boundary of cube.
The same bound holds for any ordering of the vertices.
It is known that the total number of labelled trees on $m$ vertices is $m^{m-2}$, and therefore we have
\[
	\mean{N_m(r)}  \le  \frac{n^m}{m!} m^{m-2}(\omega_d 2^d r^d)^{(m-1)}.
\]
From Stirling's approximation we have that $m! \ge  (m/e)^m$, and therefore,
\[
	\mean{N_m(r)}  \le  n^m e^m m^{-2}(\omega_d 2^dr^d)^{(m-1)} =  e \frac{n}{m^2}(e\omega_d 2^d nr^d)^{m-1}.
\]
Defining $C_2 = \frac{1}{2}(e\omega_d 2^d)^{-1}$,  if $ nr^d \le C_2(\log n)^{-\alpha}$ then 
\[
	\mean{N_m(r)} \le e\frac{n}{m^2}e^{- (m-1)(\alpha\log\log n+ \log 2)}.
\]
If $m \ge \alpha^{-1}\frac{\log n}{\log \log n} $ we therefore have (for $n$ large enough):
\[
	\mean{N_m(r)} 	\le \frac{e}{m^2},
\]
and $e / m^2 \to 0$ as $n \to \infty$.

Finally, by Markov's inequality, $\prob{N_m(r) >0} \le\mean{N_m(r)}$, and  therefore we have that $\prob{N_m(r) > 0} \to 0$ which completes the proof.
\end{proof}


With these two lemmas, we can prove the upper bound in Theorem \ref{thm:main}.


\begin{proof}[Proof of Theorem \ref{thm:main} - upper bound]

Fix a value $\alpha > 0$, and  consider two kinds of $k$-cycles: The \emph{early-born} cycles, are the ones created at a radius $r$ satisfying $nr^d \le C_2(\log n)^{-\alpha}$ (see Lemma \ref{lem:max_comp}). The \emph{late-born} cycles are all the rest.

If $\gamma\in \PH_k(n)$ is an early-born cycle, then according to Lemma \ref{lem:max_comp} it is part of a connected component with $m < \alpha^{-1}\frac{\log n}{\log\log n}$ vertices. If $\pi(\gamma) = p$, then from Lemma \ref{lem:min_comp} we  have that $C_1 p^k \le m$. Combining these two statements we have that with high probability,
\[
	\pi(\gamma) \le (C_1\alpha)^{-1/k}\param{\frac{\log n}{\log\log n}}^{1/k}.
\]
Therefore $\pi(\gamma)\le B_k \Delta_k(n)$, with $B_k = (C_1\alpha)^{-1/k}$.

Suppose now that $\gamma\in \PH_k(n)$ is a late-born cycle. This implies that $\gamma_{birth} = r$ where $nr^d  > (\log n)^{-\alpha}$, or in other words that  $\gamma_{birth} > (\frac{1}{n(\log n)^{\alpha}})^{1/d}$.
Next, in \cite{geometric} it is shown (see Theorem 6.1) that there exists $C>0$ such that if $r\ge C \left( \frac{ \log n }{n} \right)^{1/d}$ then with high probability $\cC(n,r)$ is contractible  (i.e.~can be ``shrunk" to a point, and therefore has no nontrivial cycles). In particular, this implies that $\gamma_{death} \le C \left( \frac{ \log n }{n} \right)^{1/d}$ for every cycle $\gamma$.
Thus, for late-born cycles $\gamma$
$$ \pi(\gamma) < C(\log n)^{(1+\alpha)/d} .$$
Thus, for any $\alpha < d/k - 1$, we have that with high probability the persistence of late-born cycles $\gamma$ satisfies
$$\pi(\gamma) = o \left( \left( \frac{\log n }{\log \log n} \right)^{1/k} \right).$$
\end{proof}


\section{Proof - Lower Bound} \label{sec:lower}


In this section we prove the lower bound part of Theorem \ref{thm:main} for the \cech complex $\cC(n,r)$, namely that there exists $A_k>0$ (depending on $k$ and $d$), such that with high probability, 
\[
	\Pi_k(n) \ge A_k\Delta_k(n) = A_k\param{\frac{\log n}{\log\log n}}^{1/k}.
\]
In other words, we need to show that with a high probability there exists $\gamma\in \PH_k(n)$ with $\pi(\gamma) \ge A_k\Delta_k(n)$. 

To show that, we take the unit cube $Q=[0,1]^d$ and divide it into small cubes of side $2L$. The number of small cubes we can fit in $Q$ denoted by $M$ satisfies $M\ge C_3 L^{-d}$ for some $C_3 > 0$. Denoting the small cubes by $Q_1,\ldots, Q_M$, we want to show that at least one of these cubes contains a large cycle. Let $Q_i$ be one of these cubes, and think of it as centered at the origin, so that $Q_i = [-L,L]^d$. Let $\ell < L/4$, 
denote $\hat L = \floor{L/\ell}\times \ell$, and define
\begin{align*}
	S_i^{(1)} &= [-\hat L /2 ,  \hat L/2]^{k+1} \times [-\ell/2, \ell/2]^{d-k-1} \\
	S_i^{(2)} &= [- \hat L/2 +\ell,  \hat L/2 -\ell]^{k+1} \times [-\ell/2, \ell/2]^{d-k-1} , \\
	S_i &= S_i^{(1)} \backslash S_i^{(2)}.
	\end{align*}
In other words, $S_i$ is a ``thickened" version of the boundary of a $k+1$ dimensional cube of side $\hat L\approx L$ (see Figure \ref{fig:scheme}). 

We will show that if the balls of radius $r$ around $\cP_n$ cover $S_i$ but leave most of $Q_i$ empty then $\cC(n,r)$ would have a $k$-dimensional cycle. Choosing $L$ and $\ell$ properly we can make sure that this cycle has the desirable persistence. More specifically, take $S_i$ and split it into $m$ cubes of side $\ell$, denoted by $S_{i,1},S_{i,2},\ldots, S_{i,m}$ (see Figure \ref{fig:scheme}). The number of boxes $m$ is almost proportional to the ratio of the volumes of $S_i$ and the $S_{i,j}$-s,  and therefore $m \le C_4 (L/\ell)^k$ for some $C_4>0$.
The following lemma uses the process $\cP_n$ but is in fact non-random, and provides a lower-bound to the persistence of the cycles we are looking for.


\begin{figure}[htbp]
\includegraphics[scale=0.35]{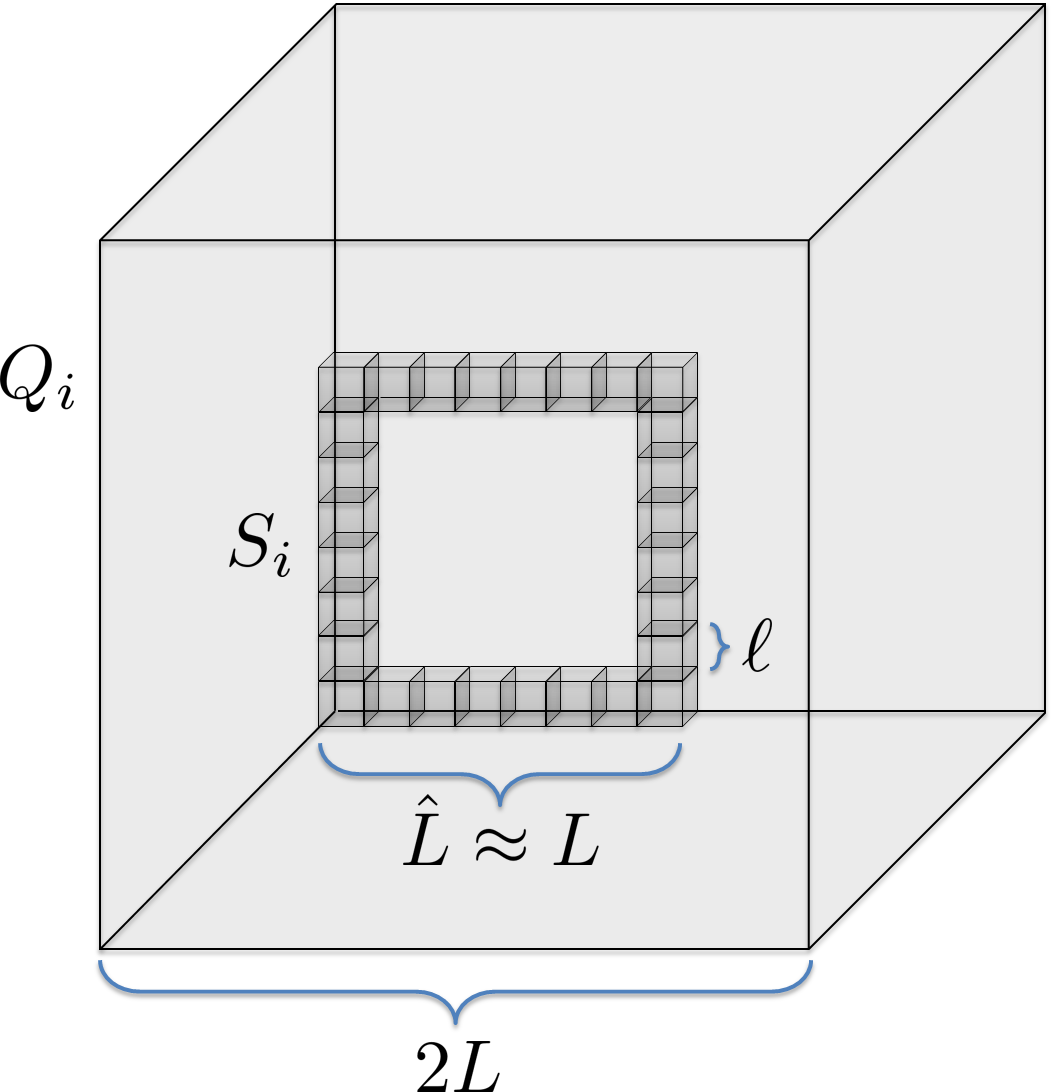}
\begin{center}
\caption{The construction we are examining to find a maximal cycle, for $d=3$ and $k=1$. $Q_i$ is the big box of side $2L$, and $S_i$ is construction made of small boxes in the middle of it, which is homotopy equivalent to a circle.}
\label{fig:scheme}
\end{center}
\end{figure}


\begin{restatable}{lem}{lemlongcyc}
\label{lem:long_cycle_1}
Suppose that for every $1\le j \le m$ we have $\abs{S_{i,j} \cap \cP_n} = 1$, and $\abs{Q_i\cap \cP_n} = m$. Then there exists $\gamma \in \PH_k(n)$ with $\pi(\gamma) \ge \frac{1}{4\sqrt d}\times\frac{L}{\ell}$.
\end{restatable}


The proof of this lemma also requires some working knowledge in algebraic topology, and therefore we postpone it to Section \ref{sec:topo}. Intuitively, the assumptions of the lemma guarantee that for every $r\in [r_1,r_2]$, where $r_1 = \sqrt{d}\ell$ and $r_2 = L/4$, the union of balls $\cU(\cP_n\cap Q_i,r)$ covers $S_i$, and is disconnected from the rest of the balls. Therefore, its shape is ``similar" to $S_i$ and forms a nontrivial $k$-cycle. Since this cycle exists through the entire range $[r_1,r_2]$ its persistence is greater than $r_2/r_1 = L/4\sqrt{d} \ell$.

Following Lemma \ref{lem:long_cycle_1}, we define the event
\[
	E_i = \set{\forall 1\le j \le m : \abs{S_{i,j} \cap \cP_n} = 1 \textrm{, and } \abs{Q_i\cap \cP_n} = m},
\]
 then $E = E_1 \cup E_2\cup\cdots\cup E_{M}$ is the event that at least one of the $Q_i$ cubes contains a large cycle.
Lemma \ref{lem:long_cycle_1} suggests that to prove there exists a large cycle it is enough to show that $E$ occurs with high probability. We start by bounding the probability of the complement event.
The next lemma shows that given the right choice of $L=L(n)$ and $\ell=\ell(n)$ we can guarantee that $E=E^{(n)}$ satisifes $\prob{E} \to 1$.


\begin{lem}\label{lem:long_cycle_exist}
Let $n\ell^d = (\log n)^{-\alpha}$ such that $\alpha > d/k$, and let $L =  \tilde A_k\Delta_k(n)\ell$ where $\tilde A_k\le \param{C_4\alpha}^{-1/k}$.
Then
\[	
\limninf \prob{E} = 1.
\]
\end{lem}


\begin{proof}

We start with the probability of $E_i$. By the spatial independence property of the Poisson process we have
\[
	\prob{E_i} = (n\ell^d)^m e^{-n(2L)^d}.
\]
and therefore, 
\[
	\prob{E^c} = \prod_{i=1}^{M} (1-\prob{E_i}) = (1-(n\ell^d)^m e^{-n(2L)^d})^{M} \le e^{-M (n\ell^d)^m e^{-n(2L)^d}}.
\]
Thus,  in order to prove that $\prob{E} \to 1$ it is enough to show that $${\mathcal E} := M (n\ell^d)^m e^{-n(2L)^d} \to \infty.$$
Recall that $M \ge C_3 L^{-d}$ and that $m \le C_4(L/\ell)^k$.
Assuming that $n\ell^d < 1$ we have,
\[
	{\mathcal E} \ge C_3 L^{-d} (n\ell^d)^{C_4(L/\ell)^k} e^{-2^dnL^d} =  C_3 L^{-d} e^{C_4(L/\ell)^k\log(n\ell^d)-2^dnL^d}
\]
Now, if $n\ell^d = (\log n)^{-\alpha}<1$ for some $\alpha >0$ and ${L}=\tilde A_k\Delta_k(n){\ell}$ for some $\tilde A_k>0$, then
\[
	nL^d = \tilde A_k^d \Delta_k^d(n)\cdot n\ell^d = \tilde A_k^d \frac{(\log n)^{d/k-\alpha}}{(\log\log n)^{d/k}}.
\]
Taking $\alpha > d/k$ yields that $nL^d\to 0$, and therefore
\[
	{\mathcal E} \ge C n \frac{(\log\log n)^{d/k}}{(\log n)^{d/k-\alpha}} e^{- C_4 \tilde A_k^k\alpha \log n},
\]
for some constant $C$. Choosing $\tilde A_k$ such that $C_4 \tilde A_k^k\alpha < 1$ we have ${\mathcal E}\to\infty$ which completes the proof.
\end{proof}

\begin{proof}[Proof of Theorem \ref{thm:main} - Lower bound]
From Lemma \ref{lem:long_cycle_exist} we have that if $n\ell^d = (\log n)^{-\alpha}$ and $L/\ell = \tilde A_k \Delta_k(n)$ then with high probability $E$ occurs. From Lemma \ref{lem:long_cycle_1}  this implies that with high probability we have a ``cubical" cycle $\gamma$ with $\pi(\gamma) \ge \tilde A_k \Delta_k(n) / 4\sqrt{d}$. Taking $A_k = \tilde A_k / 4\sqrt{d}$ completes the proof.
\end{proof}

\section{Proofs for Topological Lemmas}\label{sec:topo}

As mentioned above, the proofs for Lemmas \ref{lem:min_comp} and \ref{lem:long_cycle_1}  require some working knowledge in algebraic topology. In particular, we will be making use of the definitions of chains, cycles, boundaries and induced maps in both simplicial and singular homology. For more background, see \cite{Hatcher} or \cite{munkres_elements_1984}.  To make reading the paper fluent for readers who are less familiar with the subject, we deferred these proofs to this section. Also included in this section is the translation of Theorem \ref{thm:main} from the \cech to the Rips complex.

\subsection{Proof of Lemma \ref{lem:min_comp}}
First, we restate the lemma.

\lemmincomp*

For the sake of simplicity, we will be using homology  with coefficients in $\F = \Z/2\Z$.
Nevertheless, Lemma \ref{lem:min_comp} holds using coefficients over any field.

For every two spaces $S_1 \subset S_2$ we denote $i:S_1\hookrightarrow S_2$ as the inclusion map, and the induced map in homology will be $i_*:H_*(S_1)\to H_*(S_2)$.
For any finite set $\cP\subset [0,1]^d$ and every $r>0$, by the Nerve Lemma \ref{lem:nerve} the spaces $\cC(\cP,r)$ and $\cU(\cP,r)$ are homotopy equivalent. Therefore, there are natural maps $h: \cU(\cP, r) \to \cC(\cP, r)$ and $j: \cC(\cP, r) \to \cU(\cP, r)$ such that the induced maps $h_*: H_*(\cU(\cP, r)) \to H_*(\cC(\cP, r))$ and $j_*: H_*(\cC(\cP, r)) \to H_*(\cU(\cP, r))$ are isomorphisms.

The explicit construction of $j$ is as follows.
Each vertex in $\cC(\cP,r)$ is sent to the center of the corresponding ball. The map is then extended to every simplex by mapping it to the convex hull of the points its vertices are mapped to. Each simplex is a convex set and it is straightforward to check that in Euclidean space, the image of each simplex lies within the union of balls $\cU(\cP,r)$. This way for every $k$-simplex $\sigma\in \cC(\cP,r)$  we can define its volume $\vol_k(\sigma)$ to be the $k$-dimensional Lebesgue measure of $j(\sigma)\subset \R^d$. 

With the volume of a simplex defined, we can now define the volume of a chain. If $\gamma \in C_k(\cC(\cP,r))$ is a $k$-chain of the form $\gamma = \sum_i \alpha_i \sigma_i$ ($\alpha_i \in \set{0,1}$), then $\vol_k(\gamma) := \sum_i \alpha_i \vol_k(\sigma_i)$. In other words, the volume of a chain is defined to be the sum of the volumes of the simplexes it contains.

To prove Lemma \ref{lem:min_comp} we will be using an isoperimetric inequality related to singular cycles in $\cU(\cP,r)$ (see Theorem \ref{thm:isoper}), rather than work directly with the  simplicial cycles. To try to avoid confusion we will use $\gamma$ to refer to simplicial cycles, and $\eta$ for singular cycles. Recall that a singular $k$-simplex in $\R^d$ is a actually map $\sigma:\Delta^k\to\R^d$, where $\Delta^k$ is the standard $k$-simplex. For brevity, we will identify every singular simplex $\sigma$ with its image $\im(\sigma)\subset\R^d$, and every $k$-chain $\eta = \sum_i\alpha_i\sigma_i$ with the union $\bigcup_{i:\alpha_i\ne 0} \im(\sigma_i)\subset\R^d$.
We will also need to define the volume of a singular $k$-chain. Such a definition exists (cf. \cite{FF60}), however  we will be looking only at chains that are of the form $\eta = j(\gamma)$ where $\gamma$ is a simplicial $k$-chain in $\cC(\cP,r)$, and for those we can simply define $\vol_k(\eta) := \vol_k(\gamma)$.

Next, we define the \emph{filling radius} of a  singular $k$-cycle. Intuitively, the filling radius of a cycle measures how much we need to ``inflate" the cycle to get it filled in (so it becomes trivial). Formally,
\begin{defn}\label{def:filling}
Let $\eta$ be a compactly supported singular cycle in $\cU(\cP,r)$. A {\it filling} of $\eta$ is a $(k+1)$-chain in $\R^d$ such that $\partial \Gamma= \eta$. The {\it filling radius} $R_{fill}(\eta)$ is defined as
\[
R_{fill}(\eta) = \inf\set{\rho > 0 :  \exists \Gamma \text{  such that } \eta= \partial \Gamma \text{ and } \Gamma \subset \cU(\eta,\rho)}.
\]
In other words, $R_{fill}(\eta)$ is the smallest $\rho$ such that the ``$\rho$-thickening" of $\eta$  contains some filling $\Gamma$. 
\end{defn}

The workhorse of our proof of Lemma \ref{lem:min_comp} is the following general isoperimetric inequality due to Federer and Fleming \cite{FF60}. For a proof, see either the original article or Section 3 of Guth's expository notes on Gromov's systolic inequality \cite{Guth06}.

\begin{thm}[\bf Volume to filling radius, isoperimetric inequality] \label{thm:isoper}
Let $\eta$ be a singular $k$-cycle, such that $\vol_k(\eta) = V$. Then the filling radius of $\eta$ satisfies
\[
	R_{fill}(\eta) \le C V^{1/k},
\]
for some constant $C$ (depending on $k,d$).
\end{thm}

Recall that as in Definition~\ref{def:filling}, $\eta$ is a $k$-cycle in $\cU(\cP,r)$. However, it is worth noting that for any $k$-cycle $\gamma\in \cC(\cP,r)$, there is a canonical inclusion into $\cU(\cP,r)$. This is the geometric realization of $\eta$ (although it need not be embedded). Hence, this result also holds for cycles in the \v Cech complex.

To prove Lemma \ref{lem:min_comp} we will thus need to take two steps - (1) bound the volume of a cycle $\eta$, and (2) bound death time of $\eta$ using the filling radius $R_{fill}(\eta)$.
We start with the following definition.


\begin{defn}\label{def:R_fill}
Let $X $ be a set in $\R^d$. For $\ve >0$ the set $S$ is called an $\varepsilon$-net of $X$ if:
\begin{enumerate}
\item $S\subseteq X$
\item $X \subset \cU(S,\varepsilon)$, i.e.~ $X$ is covered by the balls of radius $\varepsilon$ around $S$, and
\item For every $p_1,p_2\in S$, $\norm{p_1-p_2}\ge \varepsilon$.
\end{enumerate}
In other words, an $\varepsilon$-net is both an $\varepsilon$-cover and an $\varepsilon$-packing.
\end{defn}


$\varepsilon$-nets are a standard construction in computational geometry and exist for any metric space ~\cite{clarkson2006nearest}. They can be constructed incrementally using the  following algorithm: (1) Initialize  $S$ to be the empty set. (2) Select any uncovered point in $X$ and add it to   $S$ (3) Mark all points of distance less than $\varepsilon$ from the selected point as covered. (4) Repeat 2-3 until there are no uncovered points. T

Next, let $\cP = \set{x_1,x_2,\ldots,x_m} \subset \R^d$ and let $S\subset \cP$ be an  $\varepsilon$-net of $\cP$. By the definition of $\ve$-nets,   the following holds:
\begin{equation}\label{eq:net_cover}
\cP\subset \cU(S,\ve)
\end{equation}
\begin{equation}\label{eq:sparse}
\norm{p_i-p_j}\geq \ve\qquad\forall p_i,p_j \in S
\end{equation} 
Using \eqref{eq:net_cover} and the triangle inequality, we also have
\begin{equation}\label{eq:cover}
\cU(\cP,\ve) \subset \cU (S,2\ve) \subset \cU(\cP,2\ve).
\end{equation} 
We will use the intermediate construction $\cU(S,2\ve)$ to bound the volume of cycles.
In particular, we will need the following lemma. We use  $[\cdot]$ to denote the equivalence class in homology of a corresponding cycle. 


\begin{lem}\label{lem:bound_inter_cycle}
Let $\cP$ and $S$ be as defined above, and let  $\gamma$ be a $k$-cycle in $\cC(S, 2\ve)$. Then $\vol_k(\gamma) \le C_5 m\ve^k$, where $C_5$ depends only on $k,d$.
Consequently, for every (singular) cycle $\eta$ in $\cU(S,2\ve)$ there exists a homologous cycle $\eta'$ such that $[\eta] = [\eta']$ and such that $\vol_k(\eta') \le C_5  m \ve^k$.
\end{lem}


\begin{proof}

The $k$-dimensional volume of $\gamma$ is the sum of the $k$-volumes of the simplexes in $\gamma$. This can be bounded by the maximal volume induced by any one simplex multiplied by the number of simplexes in $\gamma$. 

To bound the number of simplexes, first observe that $\gamma$ is supported on $S$. By \eqref{eq:sparse} every pair of vertices $p_1,p_2 \in S$ are at distance $\norm{p_1-p_2}\ge \ve$. So the balls centered at points in $S$ of radius $\ve/2$ are disjoint. This implies, by a sphere packing bound, that every vertex in $S$ has only a bounded number of neighboring vertices in $\cC(S,2\ve)$, namely the maximum number of disjoint balls of radius $\ve/2$ that can fit in a ball of radius $4\ve$.  This sphere packing number is clearly bounded above by $8^d$, the ratio of the volumes of these spheres. This implies that every vertex is contained in at most $\binom{8^d}{k}$  $k$-dimensional faces and since by assumption there are at most $m$ vertices in $\cP$ and hence $S$, there are at most $m\binom{8^d}{k}$ k-dimensional faces total. 

To bound the maximal volume of the single simplexes, observe that  the longest edge in any simplex of $\gamma$ has length at most $4\ve$. Therefore, for every simplex $\sigma$ in $\gamma$ we have $\vol_k(\sigma) \le (4 \ve)^k$ (the volume of a cube of side $4\ve$). 

To conclude, we have shown that $\gamma$ has at most $m\binom{8^d}{k}$ simplexes, the volume of each of them is bounded by $(4\ve)^k$. Therefore, $\vol_k(\gamma) \le C_5 m \ve^k$ where $C_5 = 4^k\binom{8^d}{k}$.

Next, let $\eta$ be a cycle in $\cU(\cS,2\ve)$. Since the map $j_*:H_*(\cC(S,2\ve))\to H_*(\cU(S,2\ve))$ is an isomorphism,
we can  look at the homology class $j_*^{-1}([\eta])$, and take a representative cycle $\gamma$. Defining $\eta' = j(\gamma)$ then $[\eta'] = j_*\circ j^{-1}_*([\eta]) = [\eta]$, so $\eta$ and $\eta'$ are homologous. In addition, since $\gamma$ is a cycle in $\cC(S,2\ve)$ and $\eta'=j(\gamma)$, we have that $\vol_k(\eta') = \vol_k(\gamma) \le C_5 m\ve^k$. That completes the proof.

\end{proof}

For the next lemma, consider the following sequence of maps in homology (induced by inclusion maps),

%
\begin{center}
\begin{tikzpicture}
\node (a) at (0,0 ) {$ H_k (\cU(\cP,\ve))$} ;
\node (b) at (4,0 ) {$H_k (\cU(S,2\ve))$} ;
\node (c) at (8,0 ) {$H_k (\cU(\cP,2\ve)) $} ;
%
%
\draw[->] (a) -- node [auto] {$i_*$} (b);
\draw[->] (b) -- node [auto] {$i_*$}   (c);
%
%
%
%
%
%
%
%
\end{tikzpicture}
\end{center}



\begin{lem}[\bf Vertices to volume] \label{lem:v2v} Let $\cP = \set{x_1,x_2,\ldots,x_m} \subset \R^d$. Suppose that $\eta$ is an arbitrary $k$-cycle in  $\cU(\cP, \ve)$, and let $i\circ i(\eta)$ be its image in $\cU(\cP,2\ve)$. Then there exists a $k$-cycle $\eta'$  in $\cU(\cP,2\ve)$, homologous to $i\circ i(\eta)$, such that $\vol_k(\eta') \le C_5 m\ve^k$ for some constant $C_5>0$ depending only on $k$ and $d$.
\end{lem}


\begin{proof}

Let $i(\eta)$ be the inclusion of $\eta$ into $\cU(S,2\ve)$. From Lemma \ref{lem:bound_inter_cycle} we have that there exists a cycle $\eta''$ in $\cU(S,2\ve)$ such that $[\eta''] = [i(\eta)]$ and such that $\vol_k(\eta'') \le C_5m\ve^k$. Defining $\eta' = i(\eta'')$ then $[\eta'] = i_*([\eta'']) = i_*([i(\eta)]) = [i\circ i(\eta)]$, and since the inclusion does not change the volume we have $\vol_k(\eta') = \vol_k(\eta'') \le C_5m\ve^k$. That completes the proof.

\end{proof}


Finally, we relate the filling radius to the persistence of the cycles.


\begin{lem}[\bf Filling radius to persistence] \label{lem:frtp}
If $\eta$ is a cycle in $\cU(\cP, r)$, with a filling radius $R_{fill}(\eta) = R$, then  $\eta_{death} \le R+r$.
\end{lem}


\begin{proof}
Since $\eta$ is a cycle in $\cU(\cP,r)$, then by the triangle inequality we have that $\cU(\eta, R)\subset \cU(\cP,r+R)$. By the definition of $R_{fill}$ (see Definition \ref{def:R_fill}),
this implies that there exists a $(k+1)$ cycle $\Gamma$ in $\cU(\cP,R+r)$ such that $\eta = \partial \Gamma$.
Therefore, in $\cU(\cP,R+r)$ the cycle $\eta$ is already trivial which implies that $\eta_{death} \le R+r$.
\end{proof}


We are now ready to prove Lemma \ref{lem:min_comp}.


\begin{proof}[Proof of Lemma \ref{lem:min_comp}]

Let $\gamma\in \PH_k(n)$ with $\gamma_{birth} = r$. Suppose that the simplexes constructing $\gamma$ are contained in a connected component with $m$ vertices of $\cC(n,r) = \cC(\cP_n,r)$. Let $\cP\subset \cP_n$ be the set of vertices in this connected component, then necessarily $\gamma$ is also a cycle in $\cC(\cP,r)$.

Next, take the corresponding cycle $\eta = j(\gamma)$ in $\cU(\cP,r)$.
According to Lemma \ref{lem:v2v} there exists a cycle $\eta'$ in $\cU(\cP,2r)$, homologous to $i\circ i(\eta)$, such that $\vol_k(\eta') \le C_5 m r^k$, and from Theorem \ref{thm:isoper} this implies that $R_{fill}(\eta') \le C(C_5mr^k)^{1/k} = C' m^{1/k} r$.
Using Lemma \ref{lem:frtp} we then have that $\eta'_{death} \le r(C' m^{1/k} + 2)$. Since $\eta'$ and $i\circ i(\eta)$ are homologous, then $\eta$ and $\eta'$ share the same death time, which in turn implies that $\gamma$ and $\eta'$ share the same death time. Therefore, $\pi(\gamma) \le C'm^{1/k}+2 \le C'' m^{1/k}$. In other words, if $\pi(\gamma) = p$ then we have that $p^k \le m (C'')^k$. Taking $C_1 = 1/(C'')^k$ completes the proof.
\end{proof}

\subsection{Proof of Lemma \ref{lem:long_cycle_1}}
We first restate the lemma.

\lemlongcyc*

\begin{proof}
Let $r_1 = \sqrt{d} \ell$ and $r_2 = L/4$. The assumptions that $\abs{S_{i,j} \cap \cP_n} = 1$ for every $1\le i \le m$ and $\abs{Q_i\cap \cP_n} = m$ assure that:
\begin{itemize}
\item For every $r\ge r_1$ the set  $\cU(\cP_n\cap Q_i,r)$ is connected and covers $S_i$ ;
\item For every $r\le r_2$ the sets $\cU(\cP_n\cap Q_i,r)$ and $\cU(\cP_n \backslash Q_i,r)$ are disjoint. 
\end{itemize}
In other words for every $r\in[r_1,r_2]$ the set $\cU(\cP_n\cap Q_i,r)$ is a connected component of $\cU(n,r)$. We will show that this component contains the desired cycle. 

Defining $S_i^{(r)} = \cU(S_i,r)$, for every $r\in [r_1,r_2]$ we have
\[
	S_i \subset \cU(\cP_n\cap Q_i, r) \subset S_i^{(r)}.
\]
In addition, for every $r\in[r_1,r_2]$, the inclusion $S_i\hookrightarrow S_i^{(r)}$ is a homotopy equivalence and both spaces are homotopy equivalent to a $k$-dimensional sphere, and in particular have a nontrivial $k$-cycle. A standard argument in algebraic topology (using the induced maps in homology) yields that $\cU(\cP_n \cap Q_i,r)$ must have a nontrivial $k$-cycle as well. Intuitively, since the $k$-cycle in $S_i$ ``survives" the inclusion into $S_i^{(r)}$,  it must also be present in the intermediate set $\cU(\cP_n \cap Q_i, r)$.
Now consider the following sequence induced by the inclusion maps.

\[
	H_k(S_i)\xrightarrow{i_*} H_k(\cU(\cP_n\cap Q_i, r_1))\xrightarrow{i_*} H_k(\cU(\cP_n\cap Q_i, r_2))\xrightarrow{i_*} H_k(S_i^{(r_2)})
\]
Let $\eta$ be a nontrivial cycle in $S_i$, then $i_*\circ i_*\circ i_* ([\eta])\ne 0$ since by assumption $i_*\circ i_*\circ i_*(\eta)$ is a nontrivial cycle in $S_i^{(r_2)}$ as well. Consequently, we must have $i_*([\eta])\ne 0$ and $i_*\circ i_* ([\eta])\ne 0$. Next, define $\gamma = h\circ i(\eta)$ - a cycle in $\cC(\cP_n,r_1)$, then $\gamma$ is nontrivial and so does $i(\gamma)$ in $\cC(\cP_n,r_2)$.
Therefore, $\gamma_{birth} \le r_1$ and $\gamma_{death} \ge r_2$, and then
\[
	\pi(\gamma) = \frac{\gamma_{death}}{\gamma_{birth}} \ge \frac{r_2}{r_1} = \frac{1}{4\sqrt{d}}\times\frac{L}{\ell},
\]
this completes the proof.
\end{proof}

\subsection{Proof of Theorem \ref{thm:main}  for the Vietoris-Rips Filtration}\label{sec:rips}

\begin{proof}
Let $r_2 > 2r_1$, and consider the following sequences of maps induced by the inclusions in \eqref{eq:rips_cech}.
\[
\begin{split}
	H_k(\cC(n,r_1))\xrightarrow{i_*} H_k(\cR(n,r_1)) \xrightarrow{i_*} H_k(\cR(n,r_2/2)) \xrightarrow{i_*} H_k(\cC(n,r_2)) 
\end{split}
\]
Suppose  there exists a cycle $\gamma$ in $\cC(n,r_1)$ with  $\gamma_{death}\ge r_2$. Then necessarily $i_*\circ i_*\circ i_*([\gamma])\ne 0$, which implies that both $i_*([\gamma])\ne 0$ and $i_*\circ i_*([\gamma])\ne 0$. Therefore, there exists a nontrivial cycle $\gamma' = i(\gamma)$ in $\cR(n,r_1)$ such that $\gamma'_{death} \ge r_2/2$, and consequently $\pi(\gamma') \ge r_2/2r_1$. Thus,
\begin{equation}\label{eq:CR_1}
	\prob{\Pi_k^{\cC}(n) \ge A_k\Delta_k(n)} \le \prob{\Pi_k^{\cR}(n) \ge A_k\Delta_k(n)/2}.
\end{equation}

On the other hand, we can look at the following sequence for $r_2 > 2r_1$,
\[
\begin{split}
	H_k(\cR(n,r_1))\xrightarrow{i_*} H_k(\cC(n,2r_1)) \xrightarrow{i_*} H_k(\cC(n,r_2)) \xrightarrow{i_*} H_k(\cR(n,r_2)) .
\end{split}
\]
Suppose that there exists a cycle $\gamma$ in the Rips filtration with $\gamma_{birth} \le r_1$ and $\gamma_{death}\ge r_2$. Then there exists a cycle $\gamma'$ in the \cech filtration with $\gamma'_{birth} \le 2r_1$ and $\gamma'_{death} \ge r_2$, and therefore,  $\pi(\gamma') \ge r_2/2r_1$. Thus,
\begin{equation}\label{eq:CR_2}
	\prob{\Pi_k^{\cC}(n) \le B_k\Delta_k(n)} \le \prob{\Pi_k^{\cR}(n) \le 2B_k\Delta_k(n)}.
\end{equation}

To conclude we have that
\[
	\prob{A_k \le \frac{\Pi_k^{\cC}(n)}{\Delta_k(n)} \le B_k} \le \prob{A_k/2 \le \frac{\Pi_k^{\cR}(n)}{\Delta_k(n)} \le 2B_k}.
\]
Since the left hand side converges to $1$ so does the right hand side, which completes the proof.
\end{proof}

\section{Numerical Experiments}\label{sec:experimental}
In this section, we present numerical simulations demonstrating the behavior of $\Pi_k(n)$  for the \cech complex in dimensions $d=2, 3\text{ and }4$. The experiments were run by generating a Poisson  process with rate $n$ in the unit cube of the appropriate dimension.
To generate randomness we used the standard implementation of the Mersenne Twister~\cite{mersenne_twister}. The persistence diagram computation was done using the PHAT library ~\cite{phat}.

For each sample, the \v Cech complex is built until the point of coverage (or very near coverage), since past coverage the complex is contractible and there are no changes in homology.
In dimension  2 and 3 , instead of the \v Cech filrtration, we use the $\alpha$-shape filtration~\cite{alphashape} which is based on the Delaunay triangulation. To compute the triangulations, we used the CGAL library~\cite{cgal}. The key benefit of this construction is that the simplicial complex is of a smaller size, e.g.~in 2 dimensions the size of the Delaunay triangulation is at most quadratic in the number of points. The persistence diagram are the same since for any parameter $r$, the $\alpha$-complex and \v Cech complex are homotopy equivalent (see \cite{edelsbrunner1993union}), giving rise to isomorphic homology groups.  


\begin{figure}[tbp]
\centering
\begin{subfigure}[b]{0.4\textwidth}
\includegraphics[width=\textwidth]{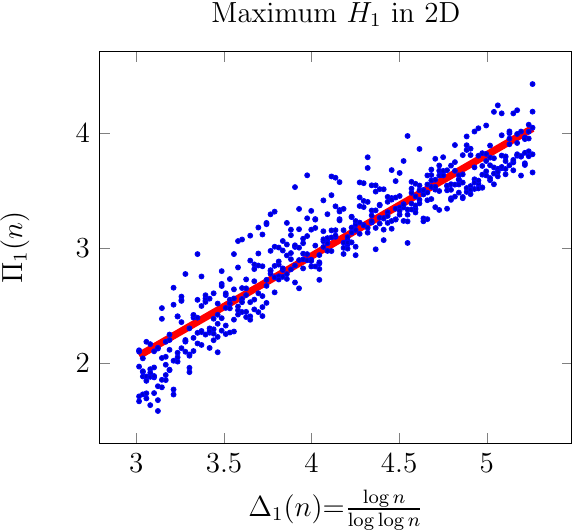}
\caption{}
\end{subfigure}
\begin{subfigure}[b]{0.4\textwidth}
\includegraphics[width=\textwidth]{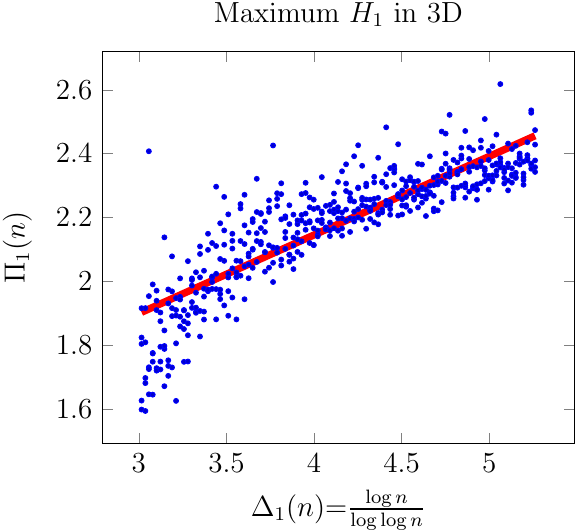}
\caption{}
\end{subfigure}

\begin{subfigure}[b]{0.4\textwidth}
\includegraphics[width=\textwidth]{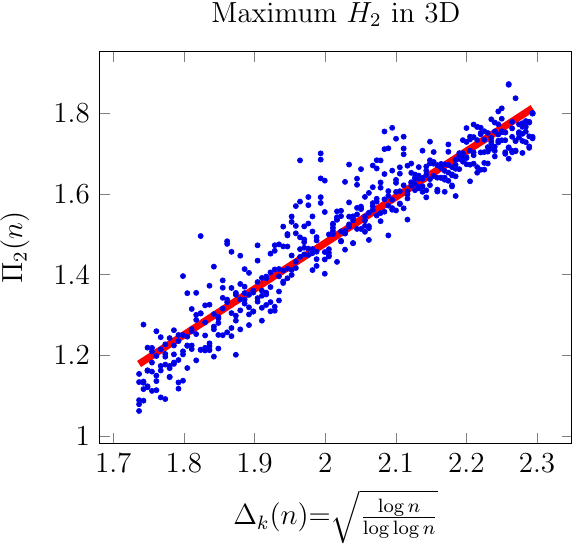}
\caption{}
\end{subfigure}
\begin{subfigure}[b]{0.4\textwidth}
\includegraphics[width=\textwidth]{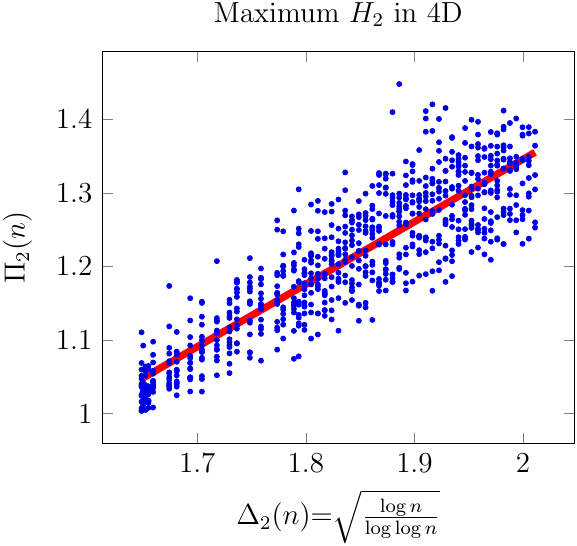}
\caption{}
\end{subfigure}
\caption{Plots of maximum persistence for the \cech filtration, against the proper scaling term $\Delta_k(n)$. We tested different dimensions for the homology and for the ambient space. (A)  $H_1$ in $\R^2$, (B) $H_1$ in $\R^3$, (C) $H_2$ in $\R^3$(D) $H_2$ in $\R^4$. Each point is the result of a different trial, and the red line represents the best linear fit. For (A), (B), and (C) the range of points is $n=10^2$ to $10^6$. For (D), the range is roughly $n=10^2$ to $10^4$. The reduced range is a consequence of computational considerations - the number of simplices grows quickly as the dimension increases.}
\label{fig:H1cech}
\end{figure}


	The results are shown in Figure~\ref{fig:H1cech}. The number of points was varied from 100 to 1,000,000  (in higher dimensions, this was reduced due to computational complexity).  We tested the behavior of $\Pi_k(n)$ for a few values of $k$, and $d$ (the ambient dimension).  For $d=2$, the only interesting case is $k=1$, namely $H_1$ (A). The resulting plot shows the maximal persistence $\Pi_1(n)$ against  $\Delta_1(n) = \log n / \log \log n$. For each value of $n$, we repeated the experiment several times.  Here, we also plot the best linear fit with the constant 0.88. We also show the results for $H_1$ when $d=3$ (B),  $H_2$ when $d=3$ (C), and $H_2$ when $d=4$ (D). We note that we performed a the same tests for the Rips filtration and the results were the same (but with different slopes).
		
There are two particularities in these plots -  the first is that  the spread is large for any one value of $n$. While it follows the straight line well it does not seem to converge to a single value. However, the resulting distributions do seem to converge, albeit slowly, as can be seen in Figure \ref{fig:dist} . The histograms  (A), (B), and (C) present the resulting ratio for 400, 2000, and 2,000,000 points, respectively. While there is a deviation, the distribution does become more concentrated around its peak. 		


 \begin{figure}[htbp]
 \begin{subfigure}[b]{0.3\textwidth}
\includegraphics[width=\textwidth]{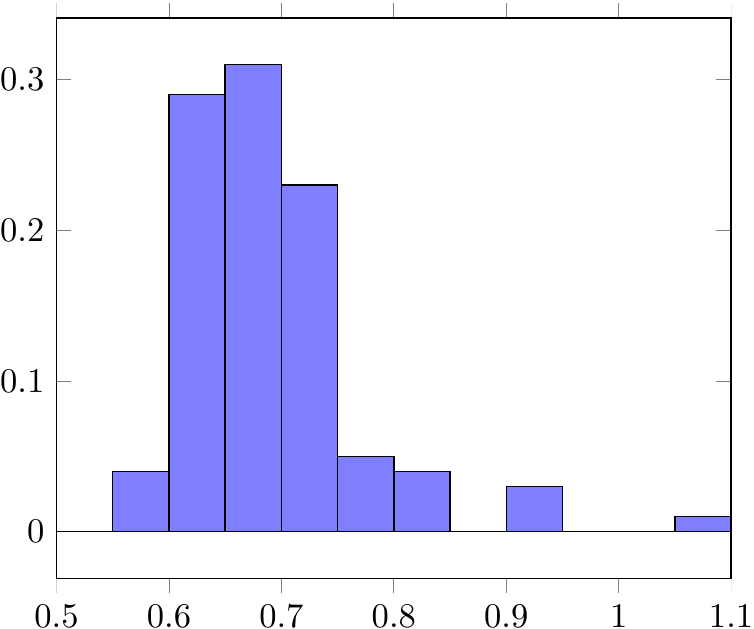}
                \caption{}
        \end{subfigure}  
         \begin{subfigure}[b]{0.3\textwidth}
\includegraphics[width=\textwidth]{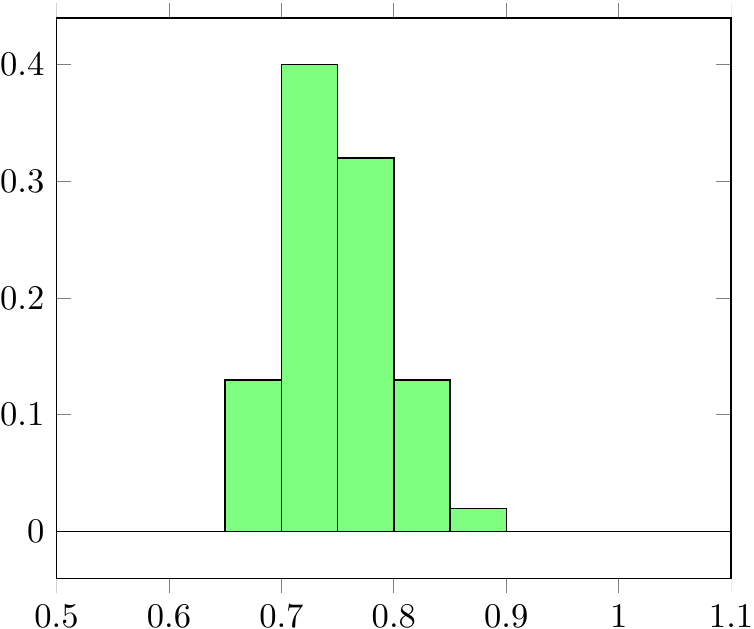}
   \caption{}
        \end{subfigure}  
 \begin{subfigure}[b]{0.3\textwidth}
\includegraphics[width=\textwidth]{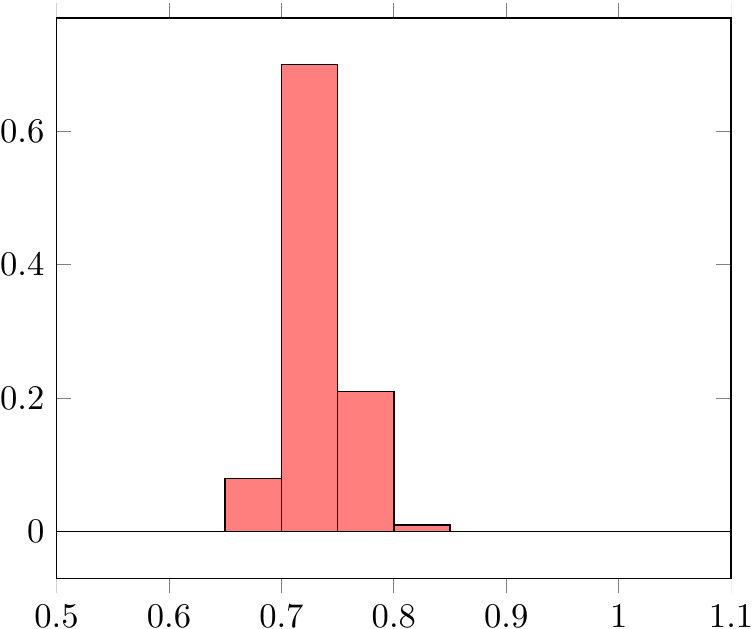}
   \caption{}
        \end{subfigure}  
\caption{Histograms of empirical $\Pi_1(n)$ in 2D normalized by $\frac{\log n}{\log \log n }$ for (A) 400 points (B) 2000 points (C) $2\times 10^6$ points.}
\label{fig:dist}
\end{figure}


The second issue is is that at smaller $n$, the maximum value drops off faster than linearly. This can be seen particularly in of Figure~\ref{fig:H1cech} (B). This phenomenon could be explained by saying that $n$ is simply not large enough for the limiting behavior to apply. Nevertheless, we tried to investigate this issue further, by considering the \v Cech complex on the flat torus ($\mathbb{T}^2$) by identifying the edges of the unit square. This part was computed using the periodic triangulations provided in CGAL \cite{cgal}. We generated points in the unit square and then computed the maximal persistence using the Euclidean metric  (e.g. the standard case) and using the metric on the flat torus. This was repeated 100 times for each value of $n$. We computed the mean and standard deviation for each value and show the results in Figure~\ref{fig:periodic}.  The red line shows the mean for the unit 
square. The red shaded region showing the interval of the mean +/- the standard deviation. The blue line (and the blue region) are the mean (and standard deviation) for the maximal persistence on the flat torus. The purple region is region where the   blue and red regions overlap. As can be seen, for most $n$ the maximal persistence is identical, indicated that the longest lived cycles did not occur near the boundary.   The difference is only visible for small  values of $n$ (where there are only a few points). At  low values of $n$, the results on the torus demonstrate a more linear behavior. This provides strong evidence that the non-linearity is due to boundary effects.

For the case of the flat torus, there are two essential one dimensional homology classes (cycles with infinite persistence) corresponding to the generators of the torus. For the above results, we ignore the essential classes.

\begin{figure}[tbp]
\centering
        \begin{subfigure}[b]{0.4\textwidth}
\includegraphics[width=\textwidth]{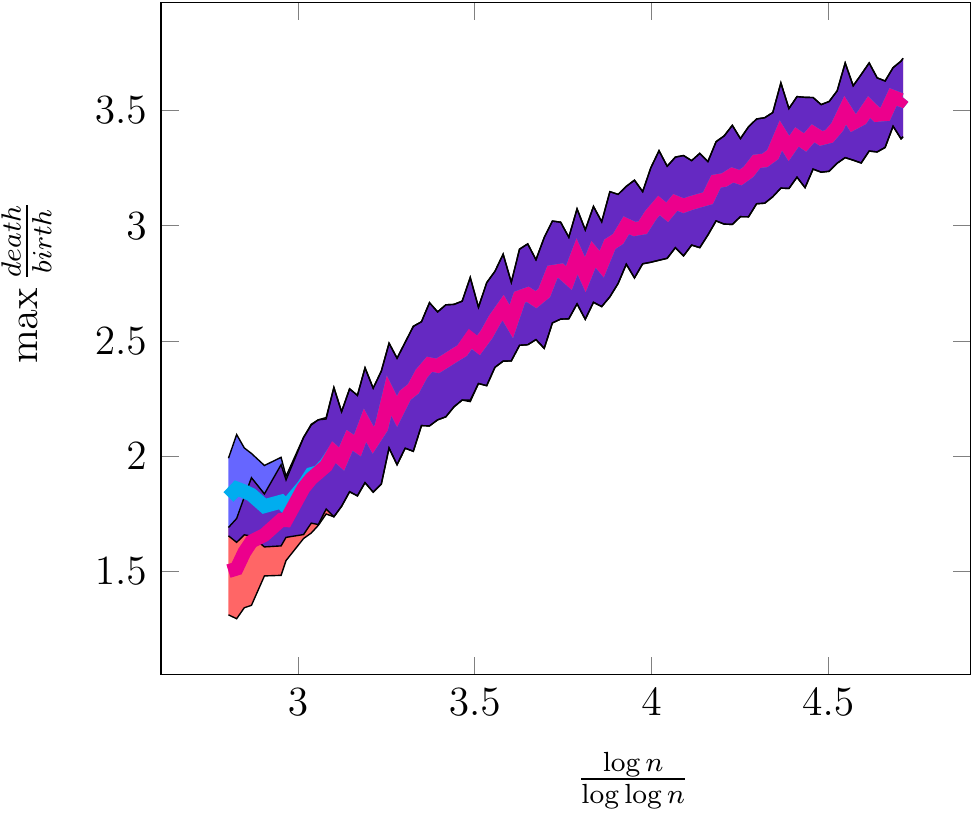}
                \caption{}
        \end{subfigure}
       \begin{subfigure}[b]{0.4\textwidth}
         
\includegraphics[width=\textwidth]{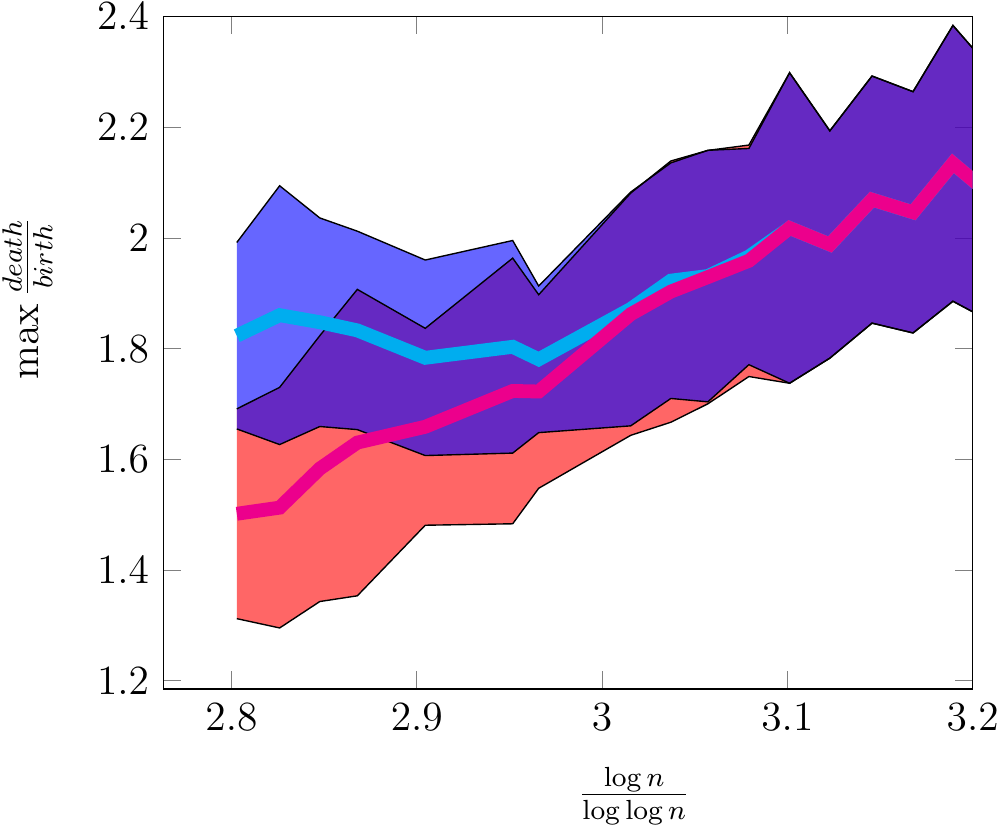}
                \caption{}
        \end{subfigure}  
      
\caption{The effect of boundaries is larger for a small number of points. The plot shows the mean maximum persistence for $H_1$ as a function of $\log n/ \log \log n$ with the shaded region showing interval corresponding to $+/-$ the standard deviation. The red line (and the red shaded region) shows the maximum persistence in the unit square, while the blue line shows the maximum persistence for the same point set in the flat torus. 
The purple region shows that for most values of $n$, the value of maximal persistence is the same in both cases. This is illustrated by an equal mean as well as the overlapping shaded regions (shown as purple). In (A), we see the plot up to several thousand points, while in (B) we show a close-up for small values of $n$, where the results differ.
}
\label{fig:periodic}
\end{figure}

\section{Conclusion}

In this paper we examined the maximum persistence of cycles in the persistent homology of either the random \cech or Rips complexes, generated by a homogeneous Poisson process in the unit cube. We showed that with a high probability we have $\Pi_k(n)\sim \param{\frac{\log n}{\log \log n}}^{1/k}$. This paper proves that upper and lower bounds exist, and it remains future work to prove stronger limiting theorems such as a law of large numbers or a central limit theorem for $\Pi_k(n)$. 

We note that while we focused on the Poisson process on the cube for simplicity, similar results can be proved with minor adjustments for non-homogeneous Poisson processes as well, and for many compact spaces other than the cube (for example, compact Riemannian manifolds). The scale of the maximum persistence should be the same ($\Delta_k(n)$), but the exact constants will be different. An important observation in this case  is that $\Pi_k(n)$ should be defined as the maximum persistence among all the ``small" cycles, i.e.~ignoring the cycles that belong to the homology of the underlying space.
Recall, that these small cycles are considered the noise in various TDA applications. Thus, revealing their distribution would be an important first step in performing noise filtering or reduction. At this point we would like to offer the following insight related to the ``signal to noise ratio" (SNR), in this kind of topological inference problems. 

Suppose that the samples are generated from a distribution on a compact manifold $\cM$, and our interest is in recovering its homology $H_k(\cM)$. The cycles that belong to the homology of $\cM$ will  show up in the \cech complex at some radius, and we can denote by $\Pi_k^{\cM}(n)$ the \emph{minimal} persistence of these cycles (in the \cech filtration). One question we might ask is - how do the signal and the noise compare? in other words - what can we say about $\Pi_k^{\cM}(n) / \Pi_k(n)$? 

The analysis we have so far already offers a preliminary answer to this question. For every cycle $\gamma$ that belongs to the homology of $\cM$ we know two things: (a) $\gamma_{death}$ is approximately constant (depending on the geometry of $\cM$), and (b) $\gamma_{birth} \le C\param{\frac{\log n}{n}}^{1/d}$ (since there are no more changes in homology past coverage, see Theorem 4.9 in \cite{BM14}). Therefore, we can conclude that
\[
	\Pi_k^{\cM}(n) \ge C' \param{\frac{n}{\log n}}^{1/d}.
\]
Combining this bound with our bound for $\Pi_k(n)$ we have for example, that for any $\eps>0$,
\[
	\frac{\Pi_k^{\cM}(n)}{\Pi_k(n)} \ge n^{1/d-\eps}.
\]
To get a better estimate for this ratio, we will need to refine our results for $\Pi_k(n)$, as well as to make more precise statements about the birth times of cycles that belong to $\cM$ (instead of using a crude upper bound). 

To conclude, we believe that the results in this paper are a promising lead in the direction of noise filtering for topological inference, and will be very useful for future analysis of probabilistic models in TDA.

\section*{Acknowledgements}
The authors would like to thank Larry Guth for helpful conversations
about isoperimetry, and to Robert Adler and Sayan Mukherjee for useful comments and fruitful discussions. 
\bibliographystyle{plain}
\bibliography{lbrefs}

\end{document}